\documentclass{article}
\usepackage{amsmath, amsthm, amssymb, amsfonts}
\usepackage[margin=1in]{geometry}
\usepackage{float}
\usepackage{color}
\usepackage{graphicx}
\graphicspath {{/}}

\newtheorem{lemma}{Lemma}
\newtheorem{theorem}{Theorem}

\newtheorem{condition}{Condition}

\newtheorem{remark}{Remark}

\begin{document}

\centerline{\sc \Large Discrete-Time Statistical Inference for Multiscale Diffusions}

\vspace{2pc}
\centerline{\sc Siragan Gailus and Konstantinos Spiliopoulos\footnote{This work has been partially supported by NSF CAREER award DMS 1550918.}}

\centerline{\sc Department of Mathematics \& Statistics, Boston University}

\centerline{\sc 111 Cummington Mall, Boston, MA 02215}

\centerline{\sc e-mail (SG): siragan@math.bu.edu (KS): kspiliop@math.bu.edu}
\vspace{2pc}

ABSTRACT. \hspace{1pc} We study statistical inference for small-noise-perturbed multiscale dynamical systems under the assumption that we observe a single time series from the slow process only. We construct estimators for both averaging and homogenization regimes, based on an appropriate misspecified model motivated by a second-order stochastic Taylor expansion of the slow process with respect to a function of the time-scale separation parameter. In the case of a fixed number of observations, we establish consistency, asymptotic normality, and asymptotic statistical efficiency of a minimum contrast estimator (MCE), the limiting variance having been identified explicitly; we furthermore establish consistency and asymptotic normality of a simplified minimum constrast estimator (SMCE), which is however not in general efficient. These results are then extended to the case of high-frequency observations under a condition restricting the rate at which the number of observations may grow vis-\`a-vis the separation of scales. Numerical simulations illustrate the theoretical results.

\section{Introduction}\label{S:introduction}
Let us consider a family of $m+(d-m)$-dimensional processes $(X^{\varepsilon},Y^{\varepsilon,})_T=\{(X^{\varepsilon}_t,Y^{\varepsilon}_t)\}_{0\leq t\leq T}$ satisfying the stochastic differential equations (SDEs)
\begin{align}
dX^{\varepsilon}_t&=\frac{\epsilon}{\delta}b_{\theta}(X^{\varepsilon}_t,Y^{\varepsilon}_t)dt
	+c_{\theta}(X^{\varepsilon}_t,Y^{\varepsilon}_t)dt
	+\sqrt\epsilon\sigma(X^{\varepsilon}_t,Y^{\varepsilon}_t)dW_t\label{model}\\
dY^{\varepsilon}_t&=\frac{\epsilon}{\delta^2}f(X^{\varepsilon}_t,Y^{\varepsilon}_t)dt
	+\frac{1}{\delta}g(X^{\varepsilon}_t,Y^{\varepsilon}_t)dt
	+\frac{\sqrt\epsilon}{\delta}\tau_1(X^{\varepsilon}_t,Y^{\varepsilon}_t)dW_t
	+\frac{\sqrt\epsilon}{\delta}\tau_2(X^{\varepsilon}_t,Y^{\varepsilon}_t)dB_t\nonumber\\
	X^{\varepsilon}_0&=x_0\in\mathcal{X}=\mathbb{R}^m,
	Y^{\varepsilon}_0=y_0\in\mathcal{Y}=\mathbb{R}^{d-m}.\nonumber
\end{align}

Here, $W$ and $B$ are independent Wiener processes and $\varepsilon=(\epsilon,\delta)\in\mathbb{R}^2_+$ satisfy $\delta=\delta(\epsilon)\to0$ as $\epsilon\to0$. Note that $\delta$ is the time-scale separation parameter while $\epsilon$ dictates the size of the noise. The vector $\theta\in\Theta\subset\mathbb{R}^{k}$ represents an unknown parameter governing the drift coefficients $b_\theta$ and $c_\theta$; the statistical problem considered herein is the estimation of $\theta$ based upon a time series sampled from a realization of the slow process $X^{\epsilon}$.

Data from physical dynamical systems commonly exhibit multiple characteristic space- or time-scales.  In addition, stochastic noise may be introduced to account for uncertainty or as an essential part of a particular modeling problem. Consequently, multiscale stochastic differential equation models like (\ref{model}) are widely employed in applied fields including physics, chemistry, and biology \cite{chauviere2010cell,janke2008rugged,zwanzig1988diffusion}, neuroscience \cite{jirsa2014nature}, meteorology \cite{majda2008applied}, and econometrics and mathematical finance \cite{jean2000derivatives,zhang2005tale} to describe stochastically-perturbed dynamical systems with two or more different space- or time-scales.

The goal of this paper is to study the problem of statistical inference for the unknown parameter vector $\theta\in\Theta\subset\mathbb{R}^{k}$ based on a fixed discrete-time sample $\{x_{t_{k}}\}_{k=1}^{n}$ from an observation of the slow component $X^{\varepsilon}$. We consider this problem in the following two regimes:
\begin{enumerate}
	\item (`$\infty$ regime,' or `homogenization regime') $\lim_{\epsilon\to0}\frac{\epsilon}{\delta}=\infty$.
	\item (`$\gamma$ regime,' or `averaging regime') $\lim_{\epsilon\to0}\frac{\epsilon}{\delta}=\gamma\in(0,\infty)$.
\end{enumerate}
In the $\infty$ regime, a standard `centering' condition (Condition \ref{centeringcondition}) is imposed to regulate the asymptotically-singular term $\frac{\epsilon}{\delta}b_\theta(X^{\varepsilon}_t,Y^{\varepsilon}_t)dt$ in the SDE for the slow component $X^{\varepsilon}$.

Statistical inference for diffusions without multiple scales (i.e., $\delta\equiv1$) has been very well studied in the literature; see for example \cite{bishwal2008parameter, KutoyantsSmallNoise, KutoyantsStatisticalInference,rao1999statistical}. In \cite{bishwal2008parameter, KutoyantsStatisticalInference,rao1999statistical}, $\epsilon\equiv\delta\equiv1$ and the asymptotic behavior of the maximum likelihood estimator (MLE) based on continuous data is studied in the time horizon limit $T\rightarrow\infty$. This is directly analogous to the limit $n\rightarrow\infty$ in the classical setting of i.i.d. observations. The case of fixed time horizon $T$ but $\epsilon\to 0$ has been studied in \cite{KutoyantsSmallNoise}, likewise based on continuous data and without multiple scales. Estimation based on discrete data without multiple scales has also been addressed in the literature; see for example \cite{guy2014parametric,SorensenUchida,Uchida}.

Maximum likelihood estimation from continuous data for multiscale models with noise of order one has been studied in \cite{azencott2010adaptive, azencott2013sub,krumscheid2013semiparametric,papavasiliou2009maximum,pavliotis2007parameter}. The authors of \cite{papavasiliou2009maximum,pavliotis2007parameter} prove that in the averaging regime, the MLE induced by the (nondeterministic) limit of the slow process $X^\varepsilon$ in (\ref{model}) with $\epsilon\equiv1$ as $\delta\to0$ is consistent under the assumption that coefficients are bounded and that the fast process $Y^\varepsilon$ takes values in a torus. It is important to point out that the regime $\varepsilon\to0$ which we study in this paper is different in that the diffusion $\sqrt\epsilon\sigma$ vanishes in the limit and, as described precisely by Theorem \ref{xlimit}, $X^\varepsilon$ converges to the solution of an ODE rather than an SDE; the (deterministic) limit does not induce a well-defined likelihood. Meanwhile, it is shown also in \cite{papavasiliou2009maximum} that direct application of the principle of maximum likelihood with discretely-sampled data via Euler-Maruyama approximation produces consistent estimates only if the data is first appropriately subsampled. Most closely related to the present work are \cite{gailusspiliopoulos, spiliopoulos2013maximum}, wherein the authors prove consistency and asymptotic normality of the continuous-data MLE for special cases of (\ref{model}).

Our focus in this paper is different. We address estimation from discrete data for multiscale diffusion models like (\ref{model}). We assume that we are given only a discrete-time sample $\{x_{t_{k}}\}_{k=1}^{n}$ from a single observation of the slow process $X^{\varepsilon}$; we assume that no data are available from the fast process $Y^{\varepsilon}$. Motivated by a second-order stochastic Taylor expansion with respect to $\sqrt{\epsilon}$, we construct a minimum contrast estimator (hereinafter referred to as the MCE) based on an appropriate misspecified model. Firstly, we prove that for any given fixed value of $n$, in either the averaging or the homogenization regime, this estimator is consistent and asymptotically normal as $\epsilon+\delta\rightarrow 0$, with a limiting variance $M(\theta; n)$ which we calculate explicitly. Knowing the limiting variance is important for statistical inference as it allows one to control the error, construct confidence intervals, and develop appropriate hypothesis tests. Going a step further, we show that $M(\theta; n)$ attains, in the limit as $n\to\infty$, the Cram\'er-Rao bound for the continuous-data estimation problem, which is to say that the estimator is asymptotically statistically efficient as first $\epsilon+\delta\to0$ and then $n\to\infty$. Secondly, we study a simplified minimum contrast estimator (hereinafter referred to as the SMCE) that can be considerably easier to apply and presents improved robustness in numerical studies. We show that the simplified estimator is also consistent and asymptotically normal, although it is not in general efficient. Thirdly, we study the behavior of the two estimators in the \textit{joint} limit $\epsilon+\delta \rightarrow 0$ \textit{and} $n\rightarrow\infty$ (the high-frequency regime), showing that consistency and asymptotic normality (and, for the MCE, asymptotic efficiency) are retained provided that the sampling interval $\Delta:=T/n$ does not decrease too quickly relative to $\epsilon$ (or equivalently, depending on the regime, to $\delta$). In particular, we prove in the high-frequency regime (a) consistency of both MCE and SMCE under the assumption that $\epsilon=o(\Delta)$, and (b) asymptotic normality of the MCE under the stronger assumption $\epsilon=o(\Delta^{2})$ but asymptotic normality of SMCE under the assumption $\epsilon=o(\Delta)$ only. It is clear that these conditions can be written also in terms of $\delta$ and $\Delta$.

To the best of our knowledge, this is the first paper on discrete-time estimation for multiscale models to describe estimators demonstrated to be consistent, asymptotically normal, and asymptotically statistically efficient. The limiting variance of the estimators is calculated explicitly. For high-frequency observations, we require that $\Delta$ not decrease too quickly relative to $\epsilon$; this is reminiscent of the subsampling prescribed in \cite{azencott2013sub, papavasiliou2009maximum}, although our case is different in that we take $T$ to be fixed, and so also the relationship between $\Delta:=T/n$ and $n$ is fixed. Despite our best efforts, we have not managed to relax these assumptions; nevertheless, numerical simulations suggest that the estimators remain well behaved, even when $\epsilon$ is not of lower order than $\Delta$ (see Tables \ref{Table5} and \ref{Table6} and Figures \ref{Fig9}-\ref{Fig12} in Section \ref{S:numericalsection}).  We emphasize that our estimators can be applied without precise knowledge of $\delta$ (which is tricky to estimate in practice) or $\epsilon$.

Let us now briefly discuss the approach that we will take. The main idea is to establish a second-order stochastic Taylor expansion of $X^{\varepsilon}$ of the form
\begin{align*}
X^{\varepsilon}=\bar X+\sqrt\epsilon\varphi^{\varepsilon},
\end{align*}
where $\bar X$ is the deterministic law-of-large-numbers limit of $X^{\varepsilon}$ and $\varphi^{\varepsilon}$ converges in distribution to a Gaussian process as $\epsilon\to0$. This representation motivates a misspecified model in which the conditional distribution of $X^{\varepsilon}_{t_{k}}$ given $X^{\varepsilon}_{t_{k-1}}=x_{t_{k-1}}$ is approximated by a Gaussian random variable having a specified mean and variance. The principle of maximum likelihood applied to this discrete, approximate process leads us to a certain contrast function, the minimizer of which we take to be our estimator. Naturally, appropriate identifiability conditions must be assumed; these conditions are however typically satisfied provided that the coefficients in the original model are sufficiently regular. We mention here that we have restricted the dependence of (\ref{model}) on the parameter $\theta$ to the coefficients $b$ and $c$ for purposes of presentation only; as inspection of the proofs will make apparent, the results of the paper continue to hold unchanged if one allows the other coefficients $\sigma, f, g, \tau_1, \tau_2$ also to depend on $\theta$.

The rest of this paper is organized as follows. Section \ref{S:modelsection} presents the main assumptions of the paper and introduces helpful notation. Section \ref{S:TaylorExpansion} develops a second-order stochastic Taylor expansion of $X^{\varepsilon}$; the representation thus obtained motivates the methods of statistical inference presented in later sections. Section \ref{S:InferenceFixedNumberObs} develops the proposed estimator assuming a fixed number $n$ of data points. We prove that the estimator is consistent and asymptotically normal as $\epsilon+\delta\rightarrow 0$, and asymptotically statistically efficient in the sense that the limiting variance, as a function of the number of observations $n$, attains the Cram\'er-Rao bound as $n\to\infty$. Section \ref{S:simplified} studies a simplification of the estimator of Section \ref{S:InferenceFixedNumberObs} that offers certain advantages in practice at the cost of an increase in limiting variance. Section \ref{S:large_n} studies the joint limit $\epsilon+\delta\to 0$ and $n\rightarrow\infty$, concluding that the proof of consistency goes through provided that one has $\epsilon \cdot n\to0$, which is to say that the squared diffusion vanishes more quickly than does the sampling interval $\Delta:=T/n$. Similarly,  asymptotic normality holds for the SMCE provided that $\epsilon=o(\Delta)$ and for the MCE if $\epsilon=o(\Delta^2)$. Section \ref{S:numericalsection} presents the results of numerical simulations to illustrate the theoretical results; simulations are described also that deliberately violate the relations $\epsilon=o(\Delta)$ and $\epsilon=o(\Delta^{2})$ that we impose for the high-frequency theory, in order to substantiate our conjecture that it may yet be possible to weaken these assumptions. Section \ref{S:Conclusions} summarizes our conclusions and discusses future directions of the research. Finally, an Appendix collects technical estimates used throughout the paper.

\section{Preliminaries and Assumptions}\label{S:modelsection}
Let us begin with a discussion of the main assumptions that we carry throughout the paper. We work with a canonical probability space $(\Omega, \mathcal{F}, P)$ equipped with a filtration $\{\mathcal{F}_t\}_{0\leq t\leq T}$ satisfying the usual conditions (namely, $\{\mathcal{F}_t\}_{0\leq t\leq T}$ is right continuous and $\mathcal{F}_0$ contains all $P$-negligible sets). Recall from (\ref{model}) that $\mathcal{X}=\mathbb{R}^m$ and $\mathcal{Y}=\mathbb{R}^{d-m}$ are, respectively, the state spaces of the slow and fast components of the dynamics.

To avoid ambiguity, we will always write $|\cdot|$ for the Frobenius (Euclidean) norm, and $||\cdot||$ for the operator norm of a matrix.

To guarantee that (\ref{model}) is well posed and that our limit results are valid, we impose the following regularity and growth conditions:
\begin{condition}\label{basicconditions} (Regularity of Coefficients)

\noindent Conditions on $c_\theta$
\begin{enumerate}
	\item $\exists K>0,q>0,r\in[0,1);\forall\theta\in\Theta,\left|c_\theta(x,y)\right|\leq K(1+|x|^r)(1+|y|^q)$.
	\item $\exists K>0,q>0;\forall\theta\in\Theta,\left|\nabla_x c_\theta(x,y)\right|+\left|\nabla_x\nabla_x c_\theta(x,y)\right|\leq K(1+|y|^q)$.
	\item $\forall\theta\in\Theta$, $c_\theta$ has two continuous derivatives in $x$, H\"older continuous in $y$ uniformly in $x$.
	\item $\forall\theta\in\Theta, \nabla_y\nabla_y c_\theta(x,y)$ is jointly continuous in $x$ and $y$.
	\item $c_\theta(x,y)$ has two locally bounded derivatives in $\theta$ with at most polynomial growth in $x$ and $y$.
\end{enumerate}

\noindent Conditions on $\sigma$
\begin{enumerate}
	\item $\forall N>0, \exists C(N);
		\forall x_1,x_2\in\mathcal{X},\forall y\in\mathcal{Y}$ with $|y|\leq N,
		|\sigma(x_1,y)-\sigma(x_2,y)|\leq C(N)|x_1-x_2|$.
	\item $\exists K>0, q>0; |\sigma(x,y)|\leq K(1+|x|^{1/2})(1+|y|^q)$.
	\item $\sigma\sigma^T$ is uniformly nondegenerate.
\end{enumerate}

\noindent Conditions on $f, \tau_1, \tau_2$
\begin{enumerate}
	\item $f, \tau_1\tau^T_1$, and $\tau_2\tau^T_2$ are twice differentiable in $x$ and $y$, the first and second derivatives in $x$ being bounded,\\
		and all partial derivatives up to second order being H\"older continuous in $y$ uniformly in $x$.
	\item $\tau_2\tau^T_2$ is uniformly nondegenerate.
\end{enumerate}

\noindent Conditions on $b_\theta$, $g$
\begin{enumerate}
	\item $b_\theta$ satisfies the same smoothness conditions as $c_\theta$.
	\item In the $\infty$ regime, $b_\theta$ and its derivatives are bounded uniformly in the first variable by polynomials in the second variable.
	\item In the $\infty$ regime, $g$ satisfies the same conditions as $c_\theta$; in the $\gamma$ regime, $g$ satsifies the same conditions as $f$.
\end{enumerate}
\end{condition}

In the limit of infinite scale separation, the slow process appears from the perspective of the fast to become `frozen.' To guarantee that the fast process has an invariant distribution when the slow process is `frozen,' we impose the following recurrence condition:

\begin{condition}\label{recurrencecondition} (Recurrence Condition)
\begin{enumerate}
	\item In the $\infty$ regime, $\lim_{|y|\to\infty}\sup_{x\in\mathcal{X}}\bigg(f(x,y)\cdot y\bigg)=-\infty$.
	\item In the $\gamma$ regime, $\lim_{|y|\to\infty}\sup_{x\in\mathcal{X}}\bigg((\gamma f+g)(x,y)\cdot y\bigg)=-\infty$.
\end{enumerate}
\end{condition}

The conditions on $f,\tau_{1},\tau_{2}$ in Condition \ref{basicconditions} and Condition \ref{recurrencecondition} guarantee that for each fixed  $x\in\mathcal{X}$ one has on $\mathcal{Y}$, in the $\infty$ regime, a unique invariant measure $\mu_{\infty,x}$ associated with the operator
\begin{align*}
	\mathcal{L}_{\infty,x}&:=f(x,\cdot)\cdot\nabla_y+
	\frac{1}{2}(\tau_1\tau^T_1+\tau_2\tau^T_2)(x,\cdot):\nabla^2_y,
\end{align*}
and in the $\gamma$ regime, a unique invariant measure $\mu_{\gamma,x}$ associated with the operator
\begin{align*}
	\mathcal{L}_{\gamma,x}&:=(\gamma f+g)(x,\cdot)\cdot\nabla_y+
	\frac{\gamma}{2}(\tau_1\tau^T_1+\tau_2\tau^T_2)(x,\cdot):\nabla^2_y.
\end{align*}
For existence of an invariant measure, the interested reader may see for example \cite{veretennikov1997}; uniqueness is a consequence of nondegenerate diffusion, as for example in \cite{pardoux2003poisson}.

In the $\infty$ regime, a standard `centering' condition is imposed to regulate the asymptotically-singular term $\frac{\epsilon}{\delta}b_{\theta}(X^{\varepsilon}_t,Y^{\varepsilon}_t)dt$ in order that we may obtain a homogenization limit.

\begin{condition}\label{centeringcondition} (Centering Condition for the $\infty$ Regime)
For each fixed $x\in\mathcal X$, $\int_{\mathcal Y}b(x,y)\mu_{\infty,x}(dy)=0$.
\end{condition}

We also need a condition on the relative rates at which $\epsilon$ and $\delta$ vanish in order to derive a suitable second-order approximation of the slow component $X^{\varepsilon}$.
\begin{condition}\label{ellcondition}
\begin{enumerate}
	\item In the $\infty$ regime, $\frac{\epsilon^{3/2}}\delta\to\ell_\infty$ as $\epsilon\to0$ for some $\ell_\infty\in(0,\infty]$.
	\item In the $\gamma$ regime, $\frac{\sqrt\epsilon}{\frac\epsilon\delta-\gamma}\to\ell_\gamma$ as $\epsilon\to0$ for some $\ell_\gamma\in(0,\infty]$.
\end{enumerate}
\end{condition}

Let us conclude this section with an introduction of notational conventions. Firstly, an asterisk in place of the regime signifier (writing $\mathcal{L}_{*,x}$, $\mu_{*,x}$, etc.) will be used when convenient if a statement is to be understood for both regimes. Next, we will in many instances wish to integrate functions of two variables $x$ and $y$ in the second variable over the invariant measures $\mu_{*,x}$ to obtain `averaged' functions of $x$ only. In such cases a bar will distinguish an averaged function from the original function; that is, for a generic function $h$ of $x$ and $y$,
\begin{align*}
\bar h(x):=\int_{\mathcal Y}h(x,y)\mu_{*,x}(dy).
\end{align*}
This bar notation should not be confused with $\bar X$, which is the first-order limit of the stochastic process $X^{\varepsilon}$, nor with $\bar\theta$, which is an estimator. Finally, we will denote by $A:B$ the Frobenius inner product $\Sigma_{i,j}[a_{i,j}\cdot b_{i,j}]$ of matrices $A=(a_{i,j})$ and $B=(b_{i,j})$.

\section{Asymptotic Behavior of $X^{\varepsilon}$}\label{S:TaylorExpansion}
In this section, we develop a stochastic Taylor expansion of $X^{\varepsilon}$, establishing in particular a representation
\begin{align}
X^{\varepsilon}=\bar X_*+\sqrt\epsilon\varphi_*^{\varepsilon}
\end{align}
in which $\bar X_*$ is deterministic and $\varphi_*^{\varepsilon}$ converges in distribution to a Gaussian process as $\epsilon\to0$. This representation is key in that it allows us to define estimators in terms of an appropriate misspecified model and prove consistency and asymptotic normality (see Sections \ref{S:InferenceFixedNumberObs}, \ref{S:simplified} and \ref{S:large_n}). We suppress the parameter $\theta$, as the results of this section hold independently of (and indeed uniformly over) $\theta\in\Theta$.

Intuitively speaking, as $\delta\to0$, the fast dynamics accelerate relative to the slow and $Y^{\varepsilon}_t$ tends to its invariant distribution at each `frozen' value $X^{\varepsilon}_t$. At the same time, the driving noise $\sqrt\epsilon\sigma dW$ of the slow process tends to $0$. It is therefore reasonable to anticipate that, in the limit, $X^{\varepsilon}$ should hew closely to the solution $\bar X_*$ of an ODE involving coefficients averaged over the fast dynamics.

Recall that in the $\infty$ regime, $\frac\epsilon\delta\to\infty$; to describe the asymptotic behavior of $X^{\varepsilon}$ in this regime, we need to find the limiting contribution of the asymptotically-singular term $\frac{\epsilon}{\delta}b(X^{\varepsilon}_t,Y^{\varepsilon}_t)dt$. It turns out that under Condition \ref{centeringcondition}, the limiting contribution may be captured in terms of the solution of a certain Poisson equation. By Theorem 3 in \cite{pardoux2003poisson}, there is a unique solution $\chi$ in the class of functions that grow at most polynomially in $|y|$ as $y\to\infty$ of the equation
\begin{align}
	\mathcal{L}_{\infty,x}\chi(x,y)&=-b(x,y),\quad \int_\mathcal{Y}\chi(x,y)\mu_{\infty,x}(dy)=0.\label{chifunction}
	\end{align}

Theorem \ref{xlimit} establishes that the averaged coefficient in the ODE for $\bar X_\infty$ is $\bar\lambda_\infty$, where
\begin{align}
\lambda_\infty(x,y):=(\nabla_y\chi\cdot g+c)(x,y).\label{Eq:lambda_infty_reg}
\end{align}

In the $\gamma$ regime, there is no singular coefficient with which to contend, and the corresponding averaged coefficient is $\bar\lambda_\gamma$, where
\begin{align}
\lambda_\gamma(x,y):=(\gamma b+c)(x,y).\label{Eq:lambda_gamma_reg}
\end{align}

Let us state this law-of-large-numbers approximation precisely in the form of a theorem. This gives a first-order approximation to $X^{\varepsilon}$ when $\epsilon,\delta$ are small.
\begin{theorem}\label{xlimit}
Assume Conditions \ref{basicconditions} and \ref{recurrencecondition} and, in the $\infty$ regime, Condition \ref{centeringcondition}; let $*$ denote the regime. For any initial condition $(x_0, y_0)\in\mathcal X\times\mathcal Y$ and $0<p<\infty$, there is a constant $\tilde K$ such that for $\epsilon$ sufficiently small,
\begin{align*}
E\sup_{0\leq t\leq T}|X^{\varepsilon}_t-\bar X_{*,t}|^p\leq\tilde K\cdot\epsilon^{p/2},
\end{align*}
where $\bar X_*$ is the (deterministic) solution of the integral equation
\begin{align*}
\bar X_{*,t}&:=x_0+\int^t_0\bar\lambda_*(\bar X_{*,s})ds,
\end{align*}
where $\bar\lambda_*$ is obtained, depending on the regime, by averaging (\ref{Eq:lambda_infty_reg}) or (\ref{Eq:lambda_gamma_reg}) over the invariant measures $\mu_{*,x}$.
\end{theorem}

Theorem \ref{xlimit} extends Theorem 1 in \cite{gailusspiliopoulos}. The proof relies on Lemma \ref{ergodiclemma} in the Appendix; we omit the details, as given
Lemma \ref{ergodiclemma}, the argument follows nearly verbatim the proof of Theorem 1 in \cite{gailusspiliopoulos}.

The essential content of Theorem \ref{xlimit} is that $X^{\varepsilon}$ tends to a deterministic first-order limit. The minimum contrast estimators that we study in this paper exploit the structure of the random fluctuations of $X^{\varepsilon}$ about this limit; to do this, we must obtain a description of the asymptotic behavior to higher order.

This, in turn, turns on quantifying more precisely the difference between the true drift $\frac\epsilon\delta b(X^{\varepsilon}_t, Y^{\varepsilon}_t)+c(X^{\varepsilon}_t, Y^{\varepsilon}_t)$ and the approximate drift $\bar\lambda_*(\bar X_t)$. The limiting coefficient $\lambda_*$ plays the role of intermediary. By Theorem 3 in \cite{pardoux2003poisson}, there is a unique solution $\Phi_*$ in the class of functions that grow at most polynomially in $|y|$ as $y\to\infty$ of the equation
\begin{align}
	\mathcal{L}_{*,x}\Phi_*(x,y)&=-\left(\lambda_*-\bar\lambda_*\right)(x,y),\quad \int_\mathcal{Y}\Phi_*(x,y)\mu_{*,x}(dy)=0.\label{phifunction}
	\end{align}
The function $\Phi$ will feature in our description of the fluctuations.

We need to define one more quantity before stating the next result. For any $0\leq s\leq t\leq T$, let $Z_{*}(t,s)$ denote the matrix-valued solution to the equation
\begin{align}
\frac{d Z_{*}(t,s)}{dt}&=(\nabla_x\bar\lambda_{*})(\bar X_{*,t})Z_{*}(t,s), \quad Z_{*}(s,s)=1_{m},\label{Eq:MatrixODE}
\end{align}
where $1_{m}$ denotes the $m\times m$ identity matrix. By Proposition 2.14 in \cite{Chicone}, the continuity of $t\mapsto (\nabla_x\bar\lambda_{*})(\bar X_{*,t})$ on $(0,T]$ guarantees the semi-group relations
$Z_{*}(t,s)=Z_{*}(t,\rho)Z_{*}(\rho,s)$ and the invertibility of $Z_{*}(t,s)$.

\begin{theorem}\label{philimit}
Assume Conditions \ref{basicconditions}, \ref{recurrencecondition}, and \ref{ellcondition}, and, in the $\infty$ regime, Condition \ref{centeringcondition}.

In the $\infty$ regime, we have a representation
\begin{align}
\frac1{\sqrt\epsilon}(X^{\varepsilon}_t-\bar X_{\infty,t})&=\frac1\ell_\infty\int^t_0Z_{\infty}(t,s)\overline{\nabla_y\Phi_\infty\cdot g}(\bar X_{\infty,s})ds+\int^t_0Z_{\infty}(t,s)(\sigma+\nabla_y\chi\cdot\tau_1)(X^{\varepsilon}_s,Y^{\varepsilon}_s)dW_s\label{phione}\\
&\hspace{4pc}+\int^t_0Z_{\infty}(t,s)\nabla_y\chi\cdot\tau_2(X^{\varepsilon}_s,Y^{\varepsilon}_s)dB_s+\tilde{\mathcal R}^\varepsilon_{\infty,t}\nonumber
\end{align}
such that for any $\eta>0$, we have $\lim_{\epsilon\to0}P\left(\sup_{0\leq t\leq T}|\tilde{\mathcal R}^\varepsilon_{\infty,t}|>\eta\right)=0$.

In the $\gamma$ regime, we have a representation
\begin{align}
\frac1{\sqrt\epsilon}(X^{\varepsilon}_t-\bar X_{\gamma,t})&=\frac1{\ell_\gamma\cdot\gamma}\int^t_0Z_{\gamma}(t,s)\overline{\gamma b-\nabla_y\Phi_\gamma\cdot g}(\bar X_{\gamma,s})ds+\int^t_0Z_{\gamma}(t,s)(\sigma+\nabla_y\Phi_\gamma\cdot\tau_1)(X^{\varepsilon}_s,Y^{\varepsilon}_s)dW_s\label{phitwo}\\
&\hspace{4pc}+\int^t_0Z_{\gamma}(t,s)\nabla_y\Phi_\gamma\cdot\tau_2(X^{\varepsilon}_s,Y^{\varepsilon}_s)dB_s+\tilde{\mathcal R}^\varepsilon_{\gamma,t}\nonumber
\end{align}
such that for any $\eta>0$, we have $\lim_{\epsilon\to0}P\left(\sup_{0\leq t\leq T}|\tilde{\mathcal R}^\varepsilon_{\gamma,t}|>\eta\right)=0$.
\end{theorem}

\begin{remark}
Before proceeding to the proof, we mention that a similar result, albeit one insufficient for our purposes, is established in \cite{spiliopoulos2014fluctuation}. It is proven there that the fluctuations process $\eta^{\varepsilon}_{t}:=\frac{X^{\varepsilon}_{t}-\bar{X}_{t}}{\sqrt{\epsilon}}$ converges weakly in the space of continuous functions $\mathcal{C}\left([0,T],\mathbb{R}^{m}\right)$
to the solution of a certain Ornstein-Uhlenbeck-type process $\eta_{t}$ (with different dynamics for the different regimes). Unfortunately, this result is of limited use in the statistical setting because the limit is in distribution, which means that prelimit and limit do not live necessarily, as such, on the same space; hence the need for Theorem \ref{philimit}, which establishes a representation in path space.
\end{remark}

\noindent\textit{Proof of Theorem \ref{philimit}.} Let us suppress the regime subscript and suppose at first that we are in the $\infty$ regime. We write
\begin{align*}
\frac1{\sqrt\epsilon}(X^{\varepsilon}_t-\bar X_t)&=I+II+III,
\end{align*}
where
\begin{align*}
I&:=\frac1{\sqrt\epsilon}\int^t_0\left(\bar\lambda(X^{\varepsilon}_s)-\bar\lambda(\bar X_s)\right)ds,\\
II&:=\frac1{\sqrt\epsilon}\int^t_0\left(\lambda(X^{\varepsilon}_s,Y^{\varepsilon}_s)-\bar\lambda(X^{\varepsilon}_s)\right)ds,\\
III&:=\int^t_0\sigma(X^{\varepsilon}_s,Y^{\varepsilon}_s)dW_s+\frac1{\sqrt\epsilon}\int^t_0\left(\frac\epsilon\delta b-\nabla_y\chi\cdot g\right)(X^{\varepsilon}_s,Y^{\varepsilon}_s)ds.
\end{align*}

We begin with an approximation of each part.
\begin{align*}
I&=\int^t_0(\nabla_x\bar\lambda)(\bar X_s)\cdot\frac1{\sqrt\epsilon}(X^{\varepsilon}_s-\bar X_s)ds+\int^t_0\left[(\nabla_x\bar\lambda)(X^{\varepsilon,\dagger}_s)-(\nabla_x\bar\lambda)(\bar X_s)\right]\cdot\frac1{\sqrt\epsilon}(X^{\varepsilon}_s-\bar X_s)ds\\
&=:\int^t_0(\nabla_x\bar\lambda)(\bar X_s)\cdot\frac1{\sqrt\epsilon}(X^{\varepsilon}_s-\bar X_s)ds+\mathcal R^\varepsilon_{I,t},
\end{align*}
where $X^{\epsilon,\dagger}_s$ is an appropriately-chosen point on the segment connecting $X^{\varepsilon}_s$ with $\bar X_s$.
\begin{align}
\mathcal R^\varepsilon_{I,t}:=\int^t_0\left[(\nabla_x\bar\lambda)(X^{\varepsilon,\dagger}_s)-(\nabla_x\bar\lambda)(\bar X_s)\right]\cdot\frac1{\sqrt\epsilon}(X^{\varepsilon}_s-\bar X_s)ds\nonumber
\end{align}
vanishes in probability uniformly in $t$ as $\epsilon\to0$; to see this, note firstly that $\int^t_0\left[(\nabla_x\bar\lambda)(X^{\epsilon,\dagger}_s)-(\nabla_x\bar\lambda)(\bar X_s)\right]ds$ vanishes in probability by compactness of $[0,T]$, continuity of $\nabla_x\bar\lambda$, and Theorem \ref{xlimit}, and secondly that $\sup_{0\leq t \leq T}\left|\frac1{\sqrt\epsilon}(X^{\varepsilon}_t-\bar X_t)\right|$ is bounded in probability by Theorem \ref{xlimit}.

Letting $\Phi$ be as in (\ref{phifunction}), applying the It\^o formula to $\Phi(x,y)$ with $(x,y)=(X^{\varepsilon}_t, Y^{\varepsilon}_t)$, and rearranging terms, we have
\begin{align}
II&=\frac\delta{\epsilon^{3/2}}\int^t_0\nabla_y\Phi\cdot g(X^{\varepsilon}_s,Y^{\varepsilon}_s)ds\label{Phireference}\\
&\hspace{4pc}+\frac\delta{\sqrt\epsilon}\int^t_0\nabla_x\Phi\cdot b(X^{\varepsilon}_s,Y^{\varepsilon}_s)ds+\frac{\delta^2}{\epsilon^{3/2}}\int^t_0\nabla_x\Phi\cdot c(X^{\varepsilon}_s,Y^{\varepsilon}_s)ds+\frac{\delta^2}{\sqrt\epsilon}\int^t_0\sigma\sigma^T:\nabla^2_x\Phi(X^{\varepsilon}_s,Y^{\varepsilon}_s)ds\nonumber\\
&\hspace{4pc}+\frac\delta{\sqrt\epsilon}\int^t_0\sigma\tau_1^T:\nabla_y\nabla_x\Phi(X^{\varepsilon}_s,Y^{\varepsilon}_s)ds+\frac{\delta^2}\epsilon\int^t_0\nabla_x\Phi\cdot\sigma(X^{\varepsilon}_s,Y^{\varepsilon}_s)dW_s\nonumber\\
&\hspace{4pc}+\frac\delta\epsilon\int^t_0(\nabla_y\Phi\cdot\tau_1,\nabla_y\Phi\cdot\tau_2)(X^{\varepsilon}_s,Y^{\varepsilon}_s)d(W,B)_s+\frac{\delta^2}{\epsilon^{3/2}}(\Phi(X^{\varepsilon}_t,Y^{\varepsilon}_t)-\Phi(x_0,y_0))\nonumber\\
&=\frac1\ell\int^t_0\overline{\nabla_y\Phi\cdot g}(\bar X_s)ds+\frac1\ell\int^t_0\left[\nabla_y\Phi\cdot g(X^{\varepsilon}_s,Y^{\varepsilon}_s)-\overline{\nabla_y\Phi\cdot g}(\bar X_s)\right]ds+\left(\frac\delta{\epsilon^{3/2}}-\frac1\ell\right)\int^t_0\nabla_y\Phi\cdot g(X^{\varepsilon}_s,Y^{\varepsilon}_s)ds\nonumber\\
&\hspace{4pc}+\frac\delta{\sqrt\epsilon}\int^t_0\nabla_x\Phi\cdot b(X^{\varepsilon}_s,Y^{\varepsilon}_s)ds+\frac{\delta^2}{\epsilon^{3/2}}\int^t_0\nabla_x\Phi\cdot c(X^{\varepsilon}_s,Y^{\varepsilon}_s)ds+\frac{\delta^2}{\sqrt\epsilon}\int^t_0\sigma\sigma^T:\nabla^2_x\Phi(X^{\varepsilon}_s,Y^{\varepsilon}_s)ds\nonumber\\
&\hspace{4pc}+\frac\delta{\sqrt\epsilon}\int^t_0\sigma\tau_1^T:\nabla_y\nabla_x\Phi(X^{\varepsilon}_s,Y^{\varepsilon}_s)ds+\frac{\delta^2}\epsilon\int^t_0\nabla_x\Phi\cdot\sigma(X^{\varepsilon}_s,Y^{\varepsilon}_s)dW_s\nonumber\\
&\hspace{4pc}+\frac\delta\epsilon\int^t_0(\nabla_y\Phi\cdot\tau_1,\nabla_y\Phi\cdot\tau_2)(X^{\varepsilon}_s,Y^{\varepsilon}_s)d(W,B)_s-
(\delta/\epsilon)(\delta/\sqrt{\epsilon})(\Phi(X^{\varepsilon}_t,Y^{\varepsilon}_t)-\Phi(x_0,y_0))\nonumber\\
&=:\frac1\ell\int^t_0\overline{\nabla_y\Phi\cdot g}(\bar X_s)ds+\mathcal R^\varepsilon_{II,t}.\nonumber
\end{align}
$\mathcal R^\varepsilon_{II}$ vanishes in probability uniformly in $t$ as $\epsilon\to0$; to see this, note firstly that

\noindent $\frac1\ell\int^t_0\left[\nabla_y\Phi\cdot g(X^{\varepsilon}_s,Y^{\varepsilon}_s)-\overline{\nabla_y\Phi\cdot g}(\bar X_s)\right]ds$ vanishes in probability by Lemma \ref{ergodiclemma}, secondly that $(\delta/\sqrt{\epsilon})(\Phi(X^{\varepsilon}_t,Y^{\varepsilon}_t)-\Phi(x_0,y_0))$ vanishes in probability by the argument of Corollary 1 in \cite{pardoux2001poisson}, and thirdly that the remaining integrals are bounded in probability by Lemma \ref{boundlemma} while their prefactors vanish.

Letting $\chi$ be as in (\ref{chifunction}), applying the It\^o formula to $\chi(x,y)$ with $(x,y)=(X^{\varepsilon}_t, Y^{\varepsilon}_t)$, and rearranging terms, we have
\begin{align*}
III&=\int^t_0(\sigma+\nabla_y\chi\cdot\tau_1)(X^{\varepsilon}_s,Y^{\varepsilon}_s)dW_s+\int^t_0\nabla_y\chi\cdot\tau_2(X^{\varepsilon}_s,Y^{\varepsilon}_s)dB_s\\
&\hspace{1pc}+\sqrt\epsilon\int^t_0\nabla_x\chi\cdot b(X^{\varepsilon}_s,Y^{\varepsilon}_s)ds+\frac\delta{\sqrt\epsilon}\int^t_0\nabla_x\chi\cdot c(X^{\varepsilon}_s,Y^{\varepsilon}_s)ds+\sqrt\epsilon\delta\int^t_0\sigma\sigma^T:\nabla^2_x\chi(X^{\varepsilon}_s,Y^{\varepsilon}_s)ds\\
&\hspace{1pc}+\sqrt\epsilon\int^t_0\sigma\tau^T_1:\nabla_y\nabla_x\chi(X^{\varepsilon}_s,Y^{\varepsilon}_s)ds+\delta\int^t_0\nabla_x\chi\cdot\sigma(X^{\varepsilon}_s,Y^{\varepsilon}_s)dW_s-(\delta/\sqrt{\epsilon})\left(\chi(X^{\varepsilon}_t,Y^{\varepsilon}_t)-\chi(x_0,y_0)\right)\\
&=:\int^t_0(\sigma+\nabla_y\chi\cdot\tau_1)(X^{\varepsilon}_s,Y^{\varepsilon}_s)dW_s+\int^t_0\nabla_y\chi\cdot\tau_2(X^{\varepsilon}_s,Y^{\varepsilon}_s)dB_s+\mathcal R^\varepsilon_{III,t}.
\end{align*}
$\mathcal R^\varepsilon_{III}$ vanishes in probability uniformly in $t$ as $\epsilon\to0$; to see this, note firstly that

\noindent $(\delta/\sqrt{\epsilon})\left(\chi(X^{\varepsilon}_t,Y^{\varepsilon}_t)-\chi(x_0,y_0)\right)$ vanishes in probability by the argument of Corollary 1 in \cite{pardoux2001poisson}, and secondly that the remaining integrals are bounded in probability by Lemma \ref{boundlemma} while their prefactors vanish.

Putting everything together and setting $\mathcal R^\varepsilon:=\mathcal R^\varepsilon_I+\mathcal R^\varepsilon_{II}+\mathcal R^\varepsilon_{III}$, we have
\begin{align*}
\frac1{\sqrt\epsilon}(X^{\varepsilon}_t-\bar X_t)&=\int^t_0(\nabla_x\bar\lambda)(\bar X_s)\cdot\frac1{\sqrt\epsilon}(X^{\varepsilon}_s-\bar X_s)ds+\frac1\ell\int^t_0\overline{\nabla_y\Phi\cdot g}(\bar X_s)ds\\
&\hspace{4pc}+\int^t_0(\sigma+\nabla_y\chi\cdot\tau_1)(X^{\varepsilon}_s,Y^{\varepsilon}_s)dW_s+\int^t_0\nabla_y\chi\cdot\tau_2(X^{\varepsilon}_s,Y^{\varepsilon}_s)dB_s+\mathcal R^\varepsilon_t.
\end{align*}

The desired representation follows; $\tilde{\mathcal R}^\varepsilon$ vanishes in probability by the same arguments as did $\mathcal{R}^\varepsilon$.

The proof for the $\gamma$ regime is similar and is therefore omitted. This concludes the proof of the theorem.

\qed

\begin{remark}
It is not difficult to see that if, for example, $\nabla_x\bar\lambda$ is globally Lipschitz, then in fact $\mathcal{R}^\varepsilon$ and $\tilde{\mathcal R}^\varepsilon$ vanish in $\mathcal L^p$. This stronger convergence is however not necessary for our purposes.
\end{remark}

\section{Statistical Inference for a fixed number of observations}\label{S:InferenceFixedNumberObs}
Let us now use Theorem \ref{philimit} to motivate the statistical estimators for the unknown parameter vector $\theta$. Suppose that the drift coefficients in the model are known functions of an unknown parameter $\theta$ which we wish to estimate using a discrete-time sample $\{x_{t_k}\}_{k=1}^n$ from the slow process; here $n>0$ is a fixed positive integer and $\{t_k\}_{k=1}^n\subset(0,T]$ is an increasing sequence of positive times. To simplify the presentation, let us assume a uniform sampling interval $\Delta:=T/n$ and adopt the convention that $t_0=0$; hence ${t_k}=k\Delta$ for $k=0, 1, $ \ldots $n$.

To motivate our estimator, recall the representation of $X^{\varepsilon}$ established in Theorem \ref{philimit}. Set, in the $\infty$ regime,
\begin{align*}
J^\theta_\infty(x,y)&:=\frac 1{\ell_\infty}(\nabla\Phi_\infty\cdot g)^\theta(x,y),\\
q^\theta_\infty(x,y)&:=[(\sigma+\nabla_y\chi\cdot\tau_1)(\sigma+\nabla_y\chi\cdot\tau_1)^T+(\nabla_y\chi\cdot\tau_2)(\nabla_y\chi\cdot\tau_2)^T]^\theta(x,y)
\end{align*}
and, in the $\gamma$ regime,
\begin{align*}
J^\theta_\gamma(x,y)&:=\frac 1{\ell_\gamma\cdot\gamma}(\gamma b-\nabla_y\Phi_\gamma\cdot g)^\theta(x,y),\\
q^\theta_\gamma(x,y)&:=[(\sigma+\nabla_y\Phi_\gamma\cdot\tau_1)(\sigma+\nabla_y\Phi_\gamma\cdot\tau_1)^T+(\nabla_y\Phi_\gamma\cdot\tau_2)(\nabla_y\Phi_\gamma\cdot\tau_2)^T]^\theta(x,y).
\end{align*}

Omitting the regime subscript and denoting by $Z^{\theta}(t,s)$ the $\theta$-dependent solution to (\ref{Eq:MatrixODE}), equations (\ref{phione}) and (\ref{phitwo}) suggest heuristically that
\begin{align}
X^{\varepsilon,\theta}_t-\bar X^\theta_t\approx\sqrt\epsilon\int^t_0Z^{\theta}(t,s)\cdot\bar J^\theta(\bar X^\theta_s)ds+\sqrt\epsilon \int^t_0Z^{\theta}(t,s)\cdot (\bar q^\theta)^{1/2}(\bar X^\theta_s)d\tilde W_s,\label{phithree}
\end{align}
where $\tilde W$ is a Wiener process. We emphasize that this approximation is not mathematically correct. Apart from ignoring the remainder term $\tilde{\mathcal{R}}^\varepsilon$, we have replaced the It\^o integral terms with a new It\^o integral matching the limiting variance. Nevertheless, proceeding from `equation' (\ref{phithree}),
\begin{align*}
&\left[X^{\varepsilon,\theta}_{t_k}-\bar X^\theta_{t_k}\right]-Z^{\theta}(t_k,t_{k-1})\cdot\left[X^{\varepsilon,\theta}_{t_{k-1}}-\bar X^\theta_{t_{k-1}}\right]\\
&\hspace{8pc}\approx\sqrt\epsilon\int^{t_k}_{t_{k-1}}Z^{\theta}(t_k,s)\cdot\bar J^\theta(\bar X^\theta_s)ds+\sqrt\epsilon \int^{t_k}_{t_{k-1}}Z^{\theta}(t_k,s)\cdot (\bar q^\theta)^{1/2}(\bar X^\theta_s)d\tilde W_s,
\end{align*}
or what is the same upon rearranging terms,
\begin{align*}
X^{\varepsilon,\theta}_{t_k}\approx\bar X^\theta_{t_k}+Z^{\theta}(t_k,t_{k-1})\cdot\left[X^{\varepsilon,\theta}_{t_{k-1}}-\bar X^\theta_{t_{k-1}}\right]+\sqrt\epsilon\int^{t_k}_{t_{k-1}}Z^{\theta}(t_k,s)\cdot\bar J^\theta(\bar X^\theta_s)ds+\sqrt\epsilon \int^{t_k}_{t_{k-1}}Z^{\theta}(t_k,s)\cdot (\bar q^\theta)^{1/2}(\bar X^\theta_s)d\tilde W_s.
\end{align*}

In this way, it would appear that for each $k$, the conditional distribution of $X^{\varepsilon}_{t_k}$ given $X^{\varepsilon}_{t_{k-1}}=x_{t_{k-1}}$ is approximated by a Gaussian distribution with mean
\begin{align}
\bar X_{{t_k}}^{\theta}+Z^{\theta}(t_k,t_{k-1})\cdot\left[x_{t_{k-1}}-\bar X_{t_{k-1}}^{\theta}\right]+\sqrt\epsilon\int^{t_k}_{t_{k-1}}Z^{\theta}(t_k,s)\cdot\bar J^{\theta}(\bar X_{s}^{\theta})ds\nonumber
\end{align}
and variance $\epsilon\int^{t_k}_{t_{k-1}}Z^{\theta}(t_k,s)\cdot \bar q^{\theta}(\bar X_{s}^{\theta})\cdot (Z^{\theta})^{T}(t_k,s)ds.$

 This heuristic Gaussian approximation to the increments may be regarded as a misspecified model,  the principle of maximum likelihood applied to which suggests the minimum contrast estimator (MCE)
\begin{align}
\bar\theta^\varepsilon(\{x_{t_k}\}_{k=1}^n):=\arg\min_{\theta\in\bar\Theta}U^\varepsilon(\theta; \{x_{t_k}\}_{k=1}^n)\label{thetabar}
\end{align}
with contrast function
\begin{align}
U^\varepsilon(\theta; \{x_{t_k}\}_{k=1}^n)&:=\Sigma_{k=1}^n \bigg[ (F^\varepsilon_k)^T(\theta; \{x_{t_k}\}_{k=1}^n) \cdot Q_k^{-1}(\theta) \cdot F^\varepsilon_k(\theta; \{x_{t_k}\}_{k=1}^n)\bigg],\label{thecontrastfunction}
\end{align}
where
\begin{align}
F^\varepsilon_k(\theta; \{x_{t_k}\}_{k=1}^n)&:=\left[[x_{t_k}-\bar X^\theta_{t_k}]-Z^{\theta}(t_k,t_{k-1})\cdot[x_{t_{k-1}}-\bar X^\theta_{t_{k-1}}]\right]\label{effdef}\\
&\hspace{2pc}-\sqrt\epsilon\int^{t_k}_{t_{k-1}}Z^{\theta}(t_k,s)\cdot\bar  J^\theta(\bar X^\theta_{s})ds,\nonumber\\
Q_k(\theta)&:=\int^{t_k}_{t_{k-1}}Z^{\theta}(t_k,s)\cdot \bar q^\theta(\bar X^\theta_{s})\cdot (Z^{\theta})^{T}(t_k,s)ds.\nonumber
\end{align}
By Lemma \ref{qlemma} in the Appendix we have that the inverse matrices $Q^{-1}_k$ exist.

The rest of this section studies regularity of the MCE when $n$ is held fixed while $\epsilon+\delta\to0$. Theorem \ref{consistency} establishes that the MCE is a consistent estimator of the true value $\theta_0$, and Theorem \ref{normality} that it is asymptotically normal, with a limiting variance $M(\theta; n)$ which we calculate explicitly. Finally, Lemma \ref{L:FisherInformation} shows that this limiting variance attains, in the limit as $n\to\infty$, the Cram\'er-Rao bound for the continuous-data estimation problem, which is to say that the estimator is asymptotically statistically efficient as first $\epsilon+\delta\to0$ and then $n\to\infty$.

Before developing the theory, we need to define the limiting contrast function
\begin{align*}
\tilde U(\theta; \{x_{t_k}\}_{k=1}^n)&:=\Sigma_{k=1}^n \bigg[ \tilde F_k^T(\theta; \{x_{t_k}\}_{k=1}^n) \cdot Q_k^{-1}(\theta) \cdot \tilde F_k(\theta; \{x_{t_k}\}_{k=1}^n)\bigg],
\end{align*}
where
\begin{align}
\tilde F_k(\theta; \{x_{t_k}\}_{k=1}^n)&:=\left[x_{t_k}-\bar X^\theta_{t_k}\right]
-Z^{\theta}(t_k,t_{k-1})\cdot\left[x_{t_{k-1}}-\bar X^\theta_{t_{k-1}}\right].\label{Eq:ReducedF}
\end{align}

By Lemma \ref{contrastlimit} in the Appendix, we have that
\begin{align*}
\lim_{\epsilon\to0}P\left(\sup_{\theta_1,\theta_2\in\Theta}|U^\varepsilon(\theta_2; \{X^{\epsilon,\theta_1}_{t_k}\}_{k=1}^n)-\tilde U(\theta_2; \{\bar X^{\theta_1}_{t_k}\}_{k=1}^n)|>\eta\right)=0.
\end{align*}

Denoting by $\theta_0$ the true value of the unknown parameter, we must assume an appropriate identifiability condition to guarantee that $\tilde U(\theta; \{\bar X^{\theta_0}_{t_k}\}_{k=1}^n)$ is uniquely minimized at $\theta=\theta_0$.

\begin{condition} (Identifiability Condition 1)\label{id1}
For any $\theta_0\in\Theta$, $\bar X^{\theta}_{t_k}=\bar X^{\theta_0}_{t_k}$ for $k=1, 2, $ \ldots $n$ if and only if $\theta=\theta_0$.
\end{condition}

\begin{remark}
Condition \ref{id1} implies that $\tilde F_k(\theta; \{\bar X^{\theta_0}\}_{k=1}^n)=0$ for $k=1, 2, $ \ldots $n$ if and only if $\theta=\theta_0$, and hence that $\tilde U(\theta; \{\bar X^{\theta_0}_{t_k}\}_{k=1}^n)$ is uniquely minimized at $\theta=\theta_0$.
\end{remark}

\begin{theorem}\label{consistency}(Consistency of the MCE) Assume Conditions \ref{basicconditions}, \ref{recurrencecondition}, \ref{ellcondition}, and \ref{id1}, and, in the $\infty$ regime, Condition \ref{centeringcondition}. For any $\theta_0\in\Theta$ and $\eta>0$,
\begin{align*}
\lim_{\epsilon\to0}P\left(|\bar\theta^\varepsilon(\{X^{\varepsilon,\theta_0}_{t_k}\}_{k=1}^n)-\theta_0|>\eta\right)=0.
\end{align*}
\end{theorem}

\noindent\textit{Proof.} Consider the modulus of continuity
\begin{align*}
w^\varepsilon(\{x_{t_k}\}_{k=1}^n,\phi):=\sup_{\theta_1,\theta_2\in\Theta; |\theta_1-\theta_2|\leq\phi}\left|U^\varepsilon(\theta_1; \{x_{t_k}\}_{k=1}^n)-U^\varepsilon(\theta_2; \{x_{t_k}\}_{k=1}^n)\right|.
\end{align*}
By Theorem 3.2.8 in \cite{dacunha}, it suffices to show that $w^\varepsilon(\{X^{\varepsilon,\theta}_{t_k}\}_{k=1}^n,\phi)$ converges in probability uniformly in $\phi$ to a function of $\phi$ that tends to $0$ as $\phi\to0$. In fact,
\begin{align*}
\lim_{\epsilon\to0}P\left(\sup_{\theta\in\Theta}\sup_{\phi>0}|w^\varepsilon(\{X^{\varepsilon,\theta}_{t_k}\}_{k=1}^n,\phi)-\tilde w^{\theta}(\phi)|>\eta\right)=0,
\end{align*}
where
\begin{align*}
\tilde w^\theta(\phi):=\sup_{\theta_1,\theta_2\in\Theta; |\theta_1-\theta_2|\leq\phi}\left|\tilde U(\theta_1; \{\bar X^{\theta}_{t_k}\}_{k=1}^n)-\tilde U(\theta_2; \{\bar X^{\theta}_{t_k}\}_{k=1}^n)\right|;
\end{align*}
this is immediate by Lemma \ref{contrastlimit} upon writing
\begin{align*}
&\Big(U^\varepsilon(\theta_1; \{X^{\varepsilon,\theta}_{t_k}\}_{k=1}^n)-U^\varepsilon(\theta_2; \{X^{\varepsilon,\theta}_{t_k}\}_{k=1}^n)\Big)-\Big(\tilde U(\theta_1; \{\bar X^{\theta}_{t_k}\}_{k=1}^n)-\tilde U(\theta_2; \{\bar X^{\theta}_{t_k}\}_{k=1}^n)\Big)\\
&\hspace{2pc}=
\Big(U^\varepsilon(\theta_1; \{X^{\varepsilon,\theta}_{t_k}\}_{k=1}^n)-\tilde U(\theta_1; \{\bar X^{\theta}_{t_k}\}_{k=1}^n)\Big)-\Big(U^\varepsilon(\theta_2; \{X^{\varepsilon,\theta}_{t_k}\}_{k=1}^n)-\tilde U(\theta_2; \{\bar X^{\theta}_{t_k}\}_{k=1}^n)\Big).
\end{align*}
That $\lim_{\phi\to0}\tilde w^\theta(\phi)=0$ is a corollary of Lemma \ref{modulus} in the Appendix, concluding the proof of the theorem.

\qed

Before proceeding to a central limit theorem for the MCE, we must assume an appropriate identifiability condition to guarantee that the variance is well behaved in the limit.

\begin{condition} (Identifiability Condition 2)\label{id2}
For any $\theta\in\Theta$, the weighted difference
\begin{align*}
Z^{\theta}(t_k,t_{k-1})\cdot\nabla_\theta\bar X^{\theta}_{t_{k-1}}-\nabla_\theta\bar X^{\theta}_{t_k}
\end{align*}
is nonzero for at least one value of $k=1, 2, $ \ldots $n$.
\end{condition}

\begin{theorem}\label{normality}(Asymptotic Normality of the MCE)
Assume Conditions \ref{basicconditions}, \ref{recurrencecondition}, \ref{ellcondition}, \ref{id1}, and \ref{id2}, and, in the $\infty$ regime, Condition \ref{centeringcondition}. For any given $\theta_0\in\Theta$,
$\frac1{\sqrt\epsilon}(\bar\theta^\varepsilon(\{X^{\varepsilon,\theta_0}_{t_k}\}^n_{k=1})-\theta_0)$ converges in distribution as $\epsilon\to0$ to the normal distribution $\mathcal N(0,M(\theta_0))$, where the covariance is given by the formula
\begin{align*}
M(\theta)&:=\left[\Sigma_{k=1}^n\bigg[\Big(Z^{\theta}(t_k,t_{k-1})\cdot\nabla_\theta\bar X^{\theta}_{t_{k-1}}-\nabla_\theta\bar X^{\theta}_{t_k}\Big)^T\cdot Q^{-1}_k(\theta)\cdot\Big(Z^{\theta}(t_k,t_{k-1})\cdot\nabla_\theta\bar X^{\theta}_{t_{k-1}}-\nabla_\theta\bar X^{\theta}_{t_k}\Big)\bigg]\right]^{-1}.\nonumber
\end{align*}
\end{theorem}

\noindent\textit{Proof.} Let us suppress the data $\{X^{\varepsilon,\theta_0}_{t_k}\}_{k=1}^n$. By Taylor's theorem,
\begin{align*}
0&=\frac1{\sqrt\epsilon}(\nabla_\theta U^\varepsilon)(\bar\theta^\varepsilon)=\frac1{\sqrt\epsilon}(\nabla_\theta U^\varepsilon)(\theta_0)+\frac1{\sqrt\epsilon}(\nabla^2_\theta U^\varepsilon)(\theta^{\varepsilon,\dagger})\cdot(\bar\theta^\varepsilon-\theta_0),
\end{align*}
where $\theta^{\varepsilon,\dagger}$ is an appropriately-chosen point on the segment connecting $\bar\theta^\varepsilon$ with $\theta_0$. Assuming the inverse exists, we may re-express this as
\begin{align}
\frac1{\sqrt\epsilon}(\bar\theta^\varepsilon-\theta_0)&=(\nabla^2_\theta U^\varepsilon)^{-1}(\theta^{\varepsilon,\dagger})\cdot\frac1{\sqrt\epsilon}(\nabla_\theta U^\varepsilon)(\theta_0).\label{reexpression}
\end{align}
Thus, it suffices to establish a limit in distribution of $\frac1{\sqrt\epsilon}(\nabla_\theta U^\varepsilon)(\theta_0)$ and an invertible limit in probability of $(\nabla^2_\theta U^\varepsilon)(\theta^{\varepsilon,\dagger})$; the interested reader is reffered to Section 3.3.4 in \cite{dacunha} for a rigorous justification of this now-classical approach.

For the weighted gradient of the contrast function, we have that
\begin{align}
\frac1{\sqrt\epsilon}(\nabla_\theta U^\varepsilon)(\theta_0)
&=\left[I+II+III+IV\right]\Big|_{\theta=\theta_0},\label{theabove}
\end{align}
where
\begin{align}
I&:=\frac2{\sqrt\epsilon}\Sigma_{k=1}^n\bigg[\Big(Z^{\theta}(t_k,t_{k-1})\cdot\nabla_\theta\bar X^{\theta}_{t_{k-1}}-\nabla_\theta\bar X^{\theta}_{t_k}\Big)^T\cdot Q^{-1}_k(\theta)\cdot F^\varepsilon_k(\theta)\bigg],\nonumber\\
II&:=-\frac2{\sqrt\epsilon}\Sigma_{k=1}^n \bigg[ \left(\nabla_\theta Z^{\theta}(t_k,t_{k-1})\cdot\left[X^{\varepsilon,\theta_0}_{t_{k-1}}-\bar X^\theta_{t_{k-1}}\right]\right)^T\cdot Q_k^{-1}(\theta) \cdot F^\varepsilon_k(\theta)\bigg],\nonumber\\
III&:=-2\Sigma_{k=1}^n\bigg[\Big(\nabla_\theta\int^{t_k}_{t_{k-1}}Z^{\theta}(t_k,s)\cdot\bar J^\theta(\bar X^\theta_s)ds\Big)^T\cdot Q^{-1}_k(\theta)\cdot F^\varepsilon_k(\theta)\bigg],\nonumber\\
IV&:=\frac1{\sqrt\epsilon}\Sigma_{k=1}^n \bigg[ (F^\varepsilon_k)^T(\theta) \cdot \nabla_\theta\left( Q_k^{-1}(\theta) \right) \cdot F^\varepsilon_k(\theta)\bigg].\nonumber
\end{align}
Note that $||\nabla_\theta Z^{\theta}(t_k,t_{k-1})||$ is uniformly bounded, that $\frac1{\sqrt\epsilon}E|X^{\varepsilon,\theta_0}_{t_{k-1}}-\bar X^{\theta_0}_{t_{k-1}}|$ is bounded for $\epsilon$ sufficiently small by Theorem \ref{xlimit}, that $||Q_k^{-1}(\theta)||$ is bounded by Lemma \ref{qlemma}, and that $E|F^\varepsilon_k(\theta)|$ vanishes at the rate of $\sqrt\epsilon$ by Lemma \ref{fbound}; putting these together, we deduce that $II\Big|_{\theta=\theta_0}$ vanishes in probability. $III\Big|_{\theta=\theta_0}$ and $IV\Big|_{\theta=\theta_0}$ similarly vanish in probability. Meanwhile, $I\Big|_{\theta=\theta_0}$ converges in distribution to $\mathcal N(0,\Gamma(\theta_0))$ with
\begin{align}
\Gamma(\theta)&=4\Sigma_{k=1}^n\bigg[\Big(Z^{\theta}(t_k,t_{k-1})\cdot\nabla_\theta\bar X^{\theta}_{t_{k-1}}-\nabla_\theta\bar X^{\theta}_{t_k}\Big)^T\cdot Q^{-1}_k(\theta)\cdot\Big(Z^{\theta}(t_k,t_{k-1})\cdot\nabla_\theta\bar X^{\theta}_{t_{k-1}}-\nabla_\theta\bar X^{\theta}_{t_k}\Big)\bigg].\label{whattosee}
\end{align}
To see this, let $\psi$ stand for $\chi$ in the $\infty$ regime and for $\Phi$ in the $\gamma$ regime. Recalling the definition (\ref{effdef}) of $F^\varepsilon_k$ we write, with $\tilde{\mathcal R}^\varepsilon$ as in Theorem \ref{philimit},
\begin{align}
&F^\varepsilon_k(\theta)=\sqrt\epsilon\int^{t_k}_{t_{k-1}}Z^{\theta}(t_k,s)(\sigma+\nabla_y\psi\cdot\tau_1)(X^{\varepsilon,\theta}_s,Y^{\epsilon,\theta}_s)dW_s\label{effref}\\
&\hspace{4pc}+\sqrt\epsilon\int^{t_k}_{t_{k-1}}Z^{\theta}(t_k,s)\nabla_y\psi\cdot\tau_2(X^{\varepsilon,\theta}_s,Y^{\epsilon,\theta}_s)dB_s+\sqrt\epsilon\left[\tilde{\mathcal R}^{\varepsilon,\theta}_{t_k}-Z^{\theta}(t_k,t_{k-1})\tilde{\mathcal R}^{\varepsilon,\theta}_{t_{k-1}}\right].\nonumber
\end{align}
Thus, we see that the sequence $\{\frac1{\sqrt\epsilon}F^\varepsilon_k(\theta)\}^n_{k=1}$ converges in distribution, as a sequence, to the sequence $\{\int^{t_k}_{t_{k-1}}Z^{\theta}(t_k,s)\bar q^{1/2}(\bar X^\theta_s)d(W,B)_s\}^n_{k=1}$ of independent Gaussian random variables, whence $I$ converges in distribution to
\begin{align}
2\Sigma_{k=1}^n \int^{t_k}_{t_{k-1}}\Big(Z^{\theta}(t_k,t_{k-1})\cdot\nabla_\theta\bar X^{\theta}_{t_{k-1}}-\nabla_\theta\bar X^{\theta}_{t_k}\Big)^T\cdot Q^{-1}_k(\theta)\cdot Z^{\theta}(t_k,s) \bar q^{1/2}(\bar X^\theta_s)d(W,B)_s,\nonumber
\end{align}
which is of course a centered Gaussian random variable with covariance matrix
\begin{align}
&4\Sigma_{k=1}^n \int^{t_k}_{t_{k-1}}\left[\Big(Z^{\theta}(t_k,t_{k-1})\cdot\nabla_\theta\bar X^{\theta}_{t_{k-1}}-\nabla_\theta\bar X^{\theta}_{t_k}\Big)^T\cdot Q^{-1}_k(\theta)\cdot Z^{\theta}(t_k,s) \bar q^{1/2}(\bar X^\theta_s)\right]\nonumber\\
&\hspace{4pc}\times\left[\Big(Z^{\theta}(t_k,t_{k-1})\cdot\nabla_\theta\bar X^{\theta}_{t_{k-1}}-\nabla_\theta\bar X^{\theta}_{t_k}\Big)^T\cdot Q^{-1}_k(\theta)\cdot Z^{\theta}(t_k,s) \bar q^{1/2}(\bar X^\theta_s)\right]^Tds\nonumber\\
&\hspace{2pc}=4\Sigma_{k=1}^n \left[\Big(Z^{\theta}(t_k,t_{k-1})\cdot\nabla_\theta\bar X^{\theta}_{t_{k-1}}-\nabla_\theta\bar X^{\theta}_{t_k}\Big)^T\cdot Q^{-1}_k(\theta)\right]\cdot Q_k(\theta)\nonumber\\
&\hspace{4pc}\times\left[\Big(Z^{\theta}(t_k,t_{k-1})\cdot\nabla_\theta\bar X^{\theta}_{t_{k-1}}-\nabla_\theta\bar X^{\theta}_{t_k}\Big)^T\cdot Q^{-1}_k(\theta)\right]^T,\nonumber
\end{align}
which is exactly $\Gamma(\theta)$. For the Hessian of the contrast function, we have, in the notation of (\ref{theabove}),
\begin{align}
(\nabla_\theta^2 U^\varepsilon)(\theta^{\varepsilon,\dagger})\label{hessianabove}
&=\sqrt\epsilon\Big[\nabla_\theta(I+II+III+IV)\Big]\bigg|_{\theta=\theta_0}.
\end{align}

The terms $\sqrt\epsilon\nabla_\theta II\Big|_{\theta=\theta_0}$, $\sqrt\epsilon\nabla_\theta III\Big|_{\theta=\theta_0}$, and $\sqrt\epsilon\nabla_\theta IV\Big|_{\theta=\theta_0}$ vanish in probability. Meanwhile, the first term may be rewritten as
\begin{align}
&\sqrt\epsilon\nabla_\theta I\Big|_{\theta=\theta_0}=2\Sigma_{k=1}^n\bigg[\Big(Z^{\theta}(t_k,t_{k-1})\cdot\nabla_\theta\bar X^{\theta}_{t_{k-1}}-\nabla_\theta\bar X^{\theta}_{t_k}\Big)^T\cdot Q^{-1}_k(\theta)\cdot\Big(Z^{\theta}(t_k,t_{k-1})\cdot\nabla_\theta\bar X^{\theta}_{t_{k-1}}-\nabla_\theta\bar X^{\theta}_{t_k}\Big)\bigg]\Bigg|_{\theta=\theta_0}\nonumber\\
&\hspace{2pc}-2\Sigma_{k=1}^n\bigg[\Big(Z^{\theta}(t_k,t_{k-1})\cdot\nabla_\theta\bar X^{\theta}_{t_{k-1}}-\nabla_\theta\bar X^{\theta}_{t_k}\Big)^T\cdot Q^{-1}_k(\theta)\cdot\left(\nabla_\theta Z^{\theta}(t_k,t_{k-1})\cdot\left[X^{\varepsilon,\theta_0}_{t_{k-1}}-\bar X^\theta_{t_{k-1}}\right]\right)\bigg]\Bigg|_{\theta=\theta_0}\nonumber\\
&\hspace{2pc}-\sqrt\epsilon\cdot 2\Sigma_{k=1}^n\bigg[\Big(Z^{\theta}(t_k,t_{k-1})\cdot\nabla_\theta\bar X^{\theta}_{t_{k-1}}-\nabla_\theta\bar X^{\theta}_{t_k}\Big)^T\cdot Q^{-1}_k(\theta)\cdot \nabla_\theta\int^{t_k}_{t_{k-1}}Z^{\theta}(t_k,s)\cdot\bar J^\theta(\bar X^\theta_s)ds\bigg]\Bigg|_{\theta=\theta_0}\nonumber\\
&\hspace{2pc}+2\Sigma_{k=1}^n\bigg[\Big(Z^{\theta}(t_k,t_{k-1})\cdot\nabla_\theta\bar X^{\theta}_{t_{k-1}}-\nabla_\theta\bar X^{\theta}_{t_k}\Big)^T\cdot \nabla_\theta\left(Q^{-1}_k(\theta)\right)\cdot F^\varepsilon_k(\theta)\bigg]\Bigg|_{\theta=\theta_0}\nonumber\\
&\hspace{2pc}+2\Sigma_{k=1}^n\bigg[\nabla_\theta\Big(Z^{\theta}(t_k,t_{k-1})\cdot\nabla_\theta\bar X^{\theta}_{t_{k-1}}-\nabla_\theta\bar X^{\theta}_{t_k}\Big)^T\cdot Q^{-1}_k(\theta)\cdot F^\varepsilon_k(\theta)\bigg]\Bigg|_{\theta=\theta_0}.\nonumber
\end{align}
All but the first summand vanish in probability by Theorem \ref{xlimit} and Lemmata \ref{fbound} and \ref{qlemma} in the Appendix; the limit in probability is therefore
\begin{align}
2\Sigma_{k=1}^n\bigg[\Big(Z^{\theta}(t_k,t_{k-1})\cdot\nabla_\theta\bar X^{\theta}_{t_{k-1}}-\nabla_\theta\bar X^{\theta}_{t_k}\Big)\cdot Q^{-1}_k(\theta)\cdot\Big(Z^{\theta}(t_k,t_{k-1})\cdot\nabla_\theta\bar X^{\theta}_{t_{k-1}}-\nabla_\theta\bar X^{\theta}_{t_k}\Big)\bigg]\Bigg|_{\theta=\theta_0},\nonumber
\end{align}
which is invertible by Lemma \ref{qlemma} and Condition \ref{id2}.

Recalling (\ref{reexpression}) and (\ref{whattosee}), we conclude that $\frac1{\sqrt\epsilon}(\bar\theta^\varepsilon(\{X^{\varepsilon,\theta_0}_{t_k}\}_{k=1}^n)-\theta_0)$ converges in distribution as $\epsilon\to0$ to $\mathcal N(0,M(\theta_0))$, with
\begin{align}
M(\theta)&=\left[\Sigma_{k=1}^n\bigg[\Big(Z^{\theta}(t_k,t_{k-1})\cdot\nabla_\theta\bar X^{\theta}_{t_{k-1}}-\nabla_\theta\bar X^{\theta}_{t_k}\Big)^T\cdot Q^{-1}_k(\theta)\cdot\Big(Z^{\theta}(t_k,t_{k-1})\cdot\nabla_\theta\bar X^{\theta}_{t_{k-1}}-\nabla_\theta\bar X^{\theta}_{t_k}\Big)\bigg]\right]^{-1},\nonumber
\end{align}
concluding the proof of the theorem.
\qed

\vspace{1pc}
We conclude this section by computing the limit $\lim_{n\to\infty}M(\theta; n)$, where the notation $M(\theta; n)=M(\theta)$ makes explicit the dependence on the number of data points $n$. A calculation shows that, in light of the ergodic theorems in the Appendix, the Fisher information established in Theorem 3.3.1 in \cite{KutoyantsParameter} leads to the asymptotic Fisher information
\begin{align}
\int^T_0\left(\nabla_\theta\bar\lambda^\theta\right)^T(\bar X^\theta_s)\cdot(\bar q^\theta)^{-1}(\bar X^\theta_s)\cdot\left(\nabla_\theta\bar\lambda^\theta\right)(\bar X^\theta_s)ds\nonumber
\end{align}
for the continuous-data estimation problem. Thus, Lemma \ref{L:FisherInformation} shows in particular that $M(\theta; n)$ attains, in the limit as $n\to\infty$, the Cram\'er-Rao bound for the continuous-data estimation problem, which is to say that the estimator is asymptotically statistically efficient as first $\epsilon+\delta\to0$ and then $n\to\infty$.

\begin{lemma} \label{L:FisherInformation}
Assume Conditions \ref{basicconditions}, \ref{recurrencecondition}, \ref{ellcondition}, \ref{id1}, and \ref{id2}, and, in the $\infty$ regime, Condition \ref{centeringcondition}. For any given $\theta\in\Theta$,
\begin{align}
\lim_{n\to\infty}M(\theta; n)=\left[\int^T_0\left(\nabla_\theta\bar\lambda^\theta\right)^T(\bar X^\theta_s)\cdot(\bar q^\theta)^{-1}(\bar X^\theta_s)\cdot\left(\nabla_\theta\bar\lambda^\theta\right)(\bar X^\theta_s)ds\right]^{-1},\nonumber
\end{align}
where $M(\theta; n)$ is the limiting variance with $n$ samples as per Theorem \ref{normality}.
\end{lemma}

\noindent\textit{Proof.} It is equivalent to show convergence of the inverses; that is, to show that
\begin{align}
\lim_{n\to\infty}M^{-1}(\theta; n)=\int^T_0\left(\nabla_\theta\bar\lambda^\theta\right)^T(\bar X^\theta_s)\cdot(\bar q^\theta)^{-1}(\bar X^\theta_s)\cdot\left(\nabla_\theta\bar\lambda^\theta\right)(\bar X^\theta_s)ds.\nonumber
\end{align}

Notice that
\begin{align}
Z^{\theta}(t_k,t_{k-1})\cdot\nabla_\theta\bar X^{\theta}_{t_{k-1}}-\nabla_\theta\bar X^{\theta}_{t_k}&=\left(Z^{\theta}(t_k,t_{k-1})-1-\int^{t_k}_{t_{k-1}}(\nabla_x\bar\lambda^{\theta})(\bar X^{\theta}_u)du\right)\cdot\nabla_\theta\bar X^{\theta}_{t_{k-1}}\label{decomposition}\\
&\hspace{2pc}+\int^{t_k}_{t_{k-1}}(\nabla_x\bar\lambda^{\theta})(\bar X^{\theta}_u)du\cdot\nabla_\theta\bar X^{\theta}_{t_{k-1}}-\nabla_\theta\left(\bar X^{\theta}_{t_{k}}-\bar X^{\theta}_{t_{k-1}}\right)\nonumber\\
&=I+II+III,\nonumber
\end{align}
where, with $1_{m}$ being the $m\times m$ identity matrix,
\begin{align}
I&:=\left(Z^{\theta}(t_k,t_{k-1})-1_{m}-\int^{t_k}_{t_{k-1}}(\nabla_x\bar\lambda^{\theta})(\bar X^{\theta}_u)du\right)\cdot\nabla_\theta\bar X^{\theta}_{t_{k-1}},\nonumber\\
II&:=\int^{t_k}_{t_{k-1}}(\nabla_x\bar\lambda^{\theta})(\bar X^{\theta}_u)\cdot\nabla_\theta\left(\bar X^{\theta}_{t_{k-1}}-\bar X^{\theta}_{u}\right)du,\nonumber\\
III&:=-\int^{t_k}_{t_{k-1}}(\nabla_\theta\bar\lambda^\theta)(\bar X^\theta_u)du.\nonumber
\end{align}
By Taylor's theorem, there is a constant $K$ such that $|I|\leq K\cdot\Delta^2$; likewise, there is a constant $K$ such that $|II|\leq K\cdot\Delta^2$. Meanwhile, the exponential terms in the definition of $Q$ converge uniformly to the identity. Finally, $III:=-\int^{t_k}_{t_{k-1}}(\nabla_\theta\bar\lambda^\theta)(\bar X^\theta_u)du$ behaves like $-\Delta\cdot(\nabla_\theta\bar\lambda^\theta)(\bar X^\theta_{t_{k-1}})$, whence we conclude that $M^{-1}(\theta; n)$ behaves asymptotically as the Riemann sum
\[
\Delta\cdot\Sigma_{k=1}^n\left[(\nabla_\theta\bar\lambda^\theta)^T\cdot (\bar q^\theta)^{-1}\cdot(\nabla_\theta\bar\lambda^\theta)\right](\bar X^\theta_{t_{k-1}}),
 \]
whose limit is of course $\int^T_0\left[\left(\nabla_\theta\bar\lambda^\theta\right)^T\cdot(\bar q^\theta)^{-1}\cdot\left(\nabla_\theta\bar\lambda^\theta\right)\right](\bar X^\theta_s)ds$, concluding the proof of the lemma.

\qed

\section{A Simplified Estimator}\label{S:simplified}
In this section we show that consistency and asymptotic normality are still achieved, even if one omits the covariance weights $Q^{-1}_k$ in the contrast function. On the one hand, one incurs by this omission an increase in the limiting variance, and asymptotic efficiency is lost (see Lemma \ref{L:ComparisonEstimators}). On the other hand, the simplification affords certain advantages - for instance, one may compute the \textit{simplified} contrast even if $\sigma$ is unknown. The simplified estimation procedure also presents enhanced robustness in numerical simulations (see Section \ref{S:numericalsection}), showing less sensitivity to the values of $\epsilon$ and $\delta$, which are commonly unknown in practice.

With the omission of the weights, one essentially imposes the additional simplification of constant diffusion on the misspecified model that gave rise to the first contrast estimator (\ref{thetabar}). Accordingly, the proofs of the results presented in this section follow near verbatim the arguments of Section \ref{S:InferenceFixedNumberObs}, and will not be presented separately in this section. The identifiability conditions remain unchanged. We also substitute from the outset the $\tilde F_k$ of equation (\ref{Eq:ReducedF}) for the $F^\varepsilon_k$ of equation (\ref{effdef}); it is not hard to see that the integral terms lost in doing so are in any event asymptotically negligible.

Let us therefore define the simplified minimum contrast estimator (SMCE)
\begin{align}
\tilde\theta(\{x_{t_k}\}_{k=1}^n):=\text{argmin}_{\theta\in\Theta}\Sigma_{k=1}^n |\tilde F_k(\theta; \{x_{t_k}\}_{k=1}^n)|^2.\label{thetatilde}
\end{align}

\begin{theorem}\label{simpleconsistency}(Consistency of the SMCE) Assume Conditions \ref{basicconditions}, \ref{recurrencecondition}, \ref{ellcondition}, \ref{id1}, and \ref{id2}, and, in the $\infty$ regime, Condition \ref{centeringcondition}. For any $\theta_0\in\Theta$ and $\eta>0$,
\begin{align*}
\lim_{\epsilon\to0}P\left(|\tilde\theta^\varepsilon(\{X^{\varepsilon,\theta_0}_{t_k}\}_{k=1}^n)-\theta_0|>\eta\right)=0.
\end{align*}
\end{theorem}

\begin{theorem}\label{simplenormality}(Asymptotic Normality of the SMCE)
Assume Conditions \ref{basicconditions}, \ref{recurrencecondition}, \ref{ellcondition}, \ref{id1}, and \ref{id2}, and, in the $\infty$ regime, Condition \ref{centeringcondition}. For any given $\theta_0\in\Theta$,
$\frac1{\sqrt\epsilon}(\tilde\theta^\varepsilon(\{X^{\varepsilon,\theta_0}_{t_k}\}^n_{k=1})-\theta_0)$ converges in distribution as $\epsilon\to0$ to the normal distribution $\mathcal N(0,\tilde M(\theta_0))$, where the covariance is given by the formula

\begin{align*}
\tilde M(\theta)&:=\Psi^{-1}(\theta)\cdot\Xi(\theta)\cdot\Psi^{-1}(\theta),
\end{align*}
where
\begin{align*}
\Psi(\theta)&:=\Sigma_{k=1}^n\bigg[\Big(Z^{\theta}(t_k,t_{k-1})\cdot\nabla_\theta\bar X^{\theta}_{t_{k-1}}-\nabla_\theta\bar X^{\theta}_{t_k}\Big)^T\cdot\Big(Z^{\theta}(t_k,t_{k-1})\cdot\nabla_\theta\bar X^{\theta}_{t_{k-1}}-\nabla_\theta\bar X^{\theta}_{t_k}\Big)\bigg],\\
\Xi(\theta)&:=\Sigma_{k=1}^n\bigg[\Big(Z^{\theta}(t_k,t_{k-1})\cdot\nabla_\theta\bar X^{\theta}_{t_{k-1}}-\nabla_\theta\bar X^{\theta}_{t_k}\Big)^T\cdot Q_k(\theta)\cdot\Big(Z^{\theta}(t_k,t_{k-1})\cdot\nabla_\theta\bar X^{\theta}_{t_{k-1}}-\nabla_\theta\bar X^{\theta}_{t_k}\Big)\bigg].
\end{align*}
\end{theorem}

\begin{lemma} \label{L:SimpleFisherInformation}
Assume Conditions \ref{basicconditions}, \ref{recurrencecondition}, \ref{ellcondition}, \ref{id1}, and \ref{id2}, and, in the $\infty$ regime, Condition \ref{centeringcondition}. For any given $\theta\in\Theta$,
\begin{align*}
\lim_{n\to\infty}\tilde M^{-1}(\theta; n)&=\check\Psi(\theta)\cdot\check\Xi^{-1}(\theta)\cdot\check\Psi(\theta),
\end{align*}
where
\begin{align*}
\check\Psi(\theta)&:=\int^T_0\left[\left(\nabla_\theta\bar\lambda^\theta\right)^T\cdot \left(\nabla_\theta\bar\lambda^\theta\right)\right](\bar X^\theta_s)ds,\\
\check\Xi(\theta)&:=\int^T_0\left[\left(\nabla_\theta\bar\lambda^\theta\right)^T\cdot\bar q^\theta\cdot\left(\nabla_\theta\bar\lambda^\theta\right)\right](\bar X^\theta_s)ds.
\end{align*}
\end{lemma}

 The next lemma establishes that the variance of the SMCE is bounded below by that of the MCE in the sense that the difference of the covariance matrices is positive semidefinite. In general, one loses asymptotic statistical efficiency when one moves from the MCE to the SMCE.
\begin{lemma}\label{L:ComparisonEstimators}
Assume Conditions \ref{basicconditions}, \ref{recurrencecondition}, \ref{ellcondition}, \ref{id1}, and \ref{id2}, and, in the $\infty$ regime, Condition \ref{centeringcondition}. Let $M(\theta; n)$ and $\tilde M(\theta; n)$ be as in Theorems \ref{normality} and \ref{simplenormality} respectively. For any $\theta\in\Theta$ and $n\geq1$, the $k\times k$ matrix
\begin{align}
\tilde M(\theta; n)- M(\theta; n)\label{covariancedelta}
\end{align}
is positive semidefinite. In particular, for any $\theta\in\Theta$, the difference of limits
\begin{align}
\lim_{n\to\infty}\tilde M(\theta; n)- \lim_{n\to\infty}M(\theta; n)\label{limitcovariancedelta}
\end{align}
is positive semidefinite.
\end{lemma}

\begin{proof}
We suppress the parameters $\theta$ and $n$ and define
\begin{align*}
d_k&:=Z(t_k,t_{k-1})\cdot\nabla_\theta\bar X_{t_{k-1}}-\nabla_\theta\bar X_{t_k}
\end{align*}
so that, recalling the notation $\Psi$ and $\Xi$ from Theorem \ref{simplenormality},
\begin{align*}
\Psi&=\Sigma_kd_k^Td_k,\\
\Xi&=\Sigma_kd_k^TQ_kd_k.
\end{align*}
The first statement, that $\tilde M - M$ is positive semidefinite, is equivalent to the statement that $\tilde M^{-1}- M^{-1}$ is negative semidefinite; we now demonstrate the latter. Let $\xi\in\mathbb{R}^k\setminus\{0\}$ be given.
\begin{align*}
|\xi^T\tilde M^{-1}\xi|^2&=|\Sigma_k\langle d_k\Xi^{-1}\Psi\xi,d_k\xi\rangle|^2\\
&=|\Sigma_k\langle (Q^{1/2})d_k\Xi^{-1}\Psi\xi,(Q^{1/2})^{-1}d_k\xi\rangle|^2\\
&\leq\left[\Sigma_k|(Q^{1/2})d_k\Xi^{-1}\Psi\xi|^2\right]\cdot\left[\Sigma_k|(Q^{1/2})^{-1}d_k\xi|^2\right]\\
&=\left[\Sigma_k\langle d_k^TQd_k\Xi^{-1}\Psi\xi,\Xi^{-1}\Psi\xi\rangle\right]\cdot\left[\Sigma_k\langle d_k^TQd_k\xi,\xi\rangle\right]\\
&=\left[\xi^T\tilde M^{-1}\xi\right]\cdot\left[\xi^T M^{-1}\xi\right].
\end{align*}
Cancelling a factor of $\xi^T\tilde M^{-1}\xi$, which is positive, one obtains
\begin{align*}
\xi^T\left[\tilde M^{-1}-M^{-1}\right]\xi\leq 0.
\end{align*}
This establishes the first statement, that (\ref{covariancedelta}) is positive semidefinite. The second statement, that the difference of limits (\ref{limitcovariancedelta}) is also positive semidefinite, follows immediately by continuity.
\end{proof}

\section{Limit of High-Frequency Observation}\label{S:large_n}
In Section \ref{S:InferenceFixedNumberObs}, we considered the important case in which the number of observations $n$ is fixed, obtaining consistent and asymptotically-normal estimators. The question naturally arises as to whether these properties carry over into the asymptotic regime of high-frequency observation, in which $n$ is taken arbitrarily large \textit{at the same time} that $\epsilon$ is taken to vanish; that is, recalling our notation $\Delta:=T/n$ for the sampling interval, we wish to consider the behavior of the estimators in the joint limit $\epsilon+\Delta\to0$.

It turns out that if $\epsilon$ is $o(\Delta)$ as $\Delta\to0$, then the theory of Sections \ref{S:InferenceFixedNumberObs} and \ref{S:simplified} carries over with only minor adjustments.

\begin{condition} (Identifiability Condition 3)\label{id3}
For any $\theta_0\in\Theta$, the integral
\begin{align*}
\Lambda_t(\theta,\theta_0):=\bar{\lambda}^{\theta_{0}}(\bar{X}^{\theta_{0}}_{t}) - \bar{\lambda}^{\theta}(\bar{X}^{\theta}_{t})-(\nabla_{x}\bar{\lambda}^{\theta}) (\bar{X}^{\theta}_{t})\cdot[\bar{X}_{t}^{\theta_{0}}- \bar{X}_{t}^{\theta}]
\end{align*}
vanishes for all $t\in[0,T]$ if and only if $\theta=\theta_0$.
\end{condition}

\begin{theorem}\label{nuconsistency} (Consistency of the MCE and SMCE as $\epsilon+\Delta\to0$)
Assume Conditions \ref{basicconditions}, \ref{recurrencecondition}, \ref{ellcondition}, and \ref{id3} and, in the $\infty$ regime, Condition \ref{centeringcondition}. Assume that $\Delta$ does not decrease too quickly relative to $\epsilon$, so that $\epsilon$ is $o(\Delta)$ as $\Delta\to0$. For any $\theta_0\in\Theta$ and $\eta>0$,
\begin{align*}
\lim_{(\epsilon+\Delta)\to0}P\left(|\bar\theta^\varepsilon(\{X^{\varepsilon,\theta_0}_{t_k}\}_{k=1}^n)-\theta_0|>\eta\right)=0,\\
\lim_{(\epsilon+\Delta)\to0}P\left(|\tilde\theta^\varepsilon(\{X^{\varepsilon,\theta_0}_{t_k}\}_{k=1}^n)-\theta_0|>\eta\right)=0.
\end{align*}
\end{theorem}

\noindent\textit{Proof.} The proof is similar to those of Theorems \ref{consistency} and \ref{simpleconsistency}, with Lemma \ref{nucontrastlimit} in place of Lemma \ref{contrastlimit}.

\qed

The following identifiability condition is the limiting analogue of Condition \ref{id2}.

\begin{condition} (Identifiability Condition 4)\label{id4}
For any $\theta\in\Theta$, the matrix-valued integral
\begin{align}
\int^T_0[(\nabla_\theta\bar\lambda^\theta)^T(\nabla_\theta\bar\lambda^\theta)](\bar X^\theta_s)ds
\end{align}
is invertible.
\end{condition}

We wish to mimic the proof of Theorem \ref{normality}. When one is considering the joint limit $\epsilon+\Delta\to0$, the contribution of the covariance corrections $Q_k^{-1}$ to the $\theta$ derivatives of the contrast function becomes a rather delicate quantity to handle. We therefore state and prove asymptotic normality of the MCE in this joint limit, only under the stronger assumption that $\epsilon$ is $o(\Delta^2)$, rather than merely $o(\Delta)$, as $\Delta\to0$. This is Theorem \ref{nunormality}.

Notably, the stronger assumption is not needed to adapt with little modification the proof of asymptotic normality of the SMCE for the same joint limit. With the omission of the weights, one essentially imposes the additional simplification of constant diffusion on the misspecified model that motivates the estimator - in particular, the corrections $Q_k^{-1}$ are essentially replaced with $1/\Delta$, and the delicate $\theta$ dependence is avoided altogether. Note that Theorem \ref{nusimplenormality} assumes only that $\epsilon$ is $o(\Delta)$.

\begin{theorem}\label{nunormality}(Asymptotic Normality of the MCE as $\epsilon+\Delta\to0$)
Assume Conditions \ref{basicconditions}, \ref{recurrencecondition}, \ref{ellcondition}, \ref{id3}, and \ref{id4}, and, in the $\infty$ regime, Condition \ref{centeringcondition}. Assume that $\Delta$ decreases slowly relative to $\epsilon$, so that $\epsilon$ is $o(\Delta^2)$ as $\Delta\to0$. For any given $\theta_0\in\Theta$,
$\frac1{\sqrt\epsilon}(\bar\theta^\varepsilon(\{X^{\varepsilon,\theta_0}_{t_k}\}^n_{k=1})-\theta_0)$ converges in distribution as $\epsilon+\Delta\to0$ to the normal distribution $\mathcal N(0,M(\theta_0))$, where the covariance is given by the formula
\begin{align*}
M(\theta)&:=\left[\int^T_0\left(\nabla_\theta\bar\lambda^\theta\right)^T(\bar X^\theta_s)\cdot(\bar q^\theta)^{-1}(\bar X^\theta_s)\cdot\left(\nabla_\theta\bar\lambda^\theta\right)(\bar X^\theta_s)ds\right]^{-1}.\nonumber
\end{align*}

\end{theorem}

\noindent\textit{Proof.} Let us examine the proof of Theorem \ref{normality}. Consider, in the notation of (\ref{theabove}), the term $II$.
It is easy to see that $\nabla_\theta Z^{\theta}(t_k,t_{k-1})$ is uniformly bounded by a constant times $\Delta$, and Lemma \ref{qlemma} establishes that $||Q_k^{-1}(\theta)||$ is bounded by a constant times $1/\Delta$. Meanwhile, $E|\frac1{\sqrt\epsilon}\left[X^{\varepsilon,\theta_0}_{t_{k-1}}-\bar X^\theta_{t_{k-1}}\right]|^2$ is uniformly bounded by Theorem \ref{xlimit}. Thus, by an application of H\"older's inequality, $II$ will vanish even as $n\to\infty$, provided that the general term $E|F^\varepsilon_k(\theta)|^2$ vanishes faster than $\Delta^2$; but this is immediate from Lemma \ref{nufbound}, since we have assumed that $\epsilon$ tends to $0$ faster than does $\Delta$. $III$ vanishes similarly. $IV$ vanishes by Lemma \ref{nufbound} and the fact that $||\nabla_\theta(Q_k^{-1}(\theta))||$ is bounded by a constant times $1/\Delta$ (it is in order to accommodate this somewhat na\"ive bound on the derivative of the covariance correction that we have assumed for this theorem that $\epsilon$ is $o(\Delta^2)$, rather than merely $o(\Delta)$). The convergence in distribution of $I$ is as before, although in the regime of joint asymptotics $\epsilon+\Delta\to0$, one must of course combine the arguments of Theorem \ref{normality} and Lemma \ref{L:FisherInformation}.

The convergence in probability of the Hessian (\ref{hessianabove}) even as $n\to\infty$ is likewise.

\qed

\begin{theorem}\label{nusimplenormality}(Asymptotic Normality of the SMCE as $\epsilon+\Delta\to0$)
Assume Conditions \ref{basicconditions}, \ref{recurrencecondition}, \ref{ellcondition}, \ref{id3}, and \ref{id4}, and, in the $\infty$ regime, Condition \ref{centeringcondition}. Assume that $\Delta$ does not decrease too quickly relative to $\epsilon$, so that $\epsilon$ is $o(\Delta)$ as $\Delta\to0$. For any given $\theta_0\in\Theta$,
$\frac1{\sqrt\epsilon}(\tilde\theta^\varepsilon(\{X^{\varepsilon,\theta_0}_{t_k}\}^n_{k=1})-\theta_0)$ converges in distribution as $\epsilon+\Delta\to0$ to the normal distribution $\mathcal N(0,\tilde M(\theta_0))$, where the covariance is given by the formula

\begin{align*}
\tilde M(\theta)&:=\check\Psi^{-1}(\theta)\cdot\check\Xi(\theta)\cdot\check\Psi^{-1}(\theta),
\end{align*}
where
\begin{align*}
\check\Psi(\theta)&:=\int^T_0\left[\left(\nabla_\theta\bar\lambda^\theta\right)^T\cdot \left(\nabla_\theta\bar\lambda^\theta\right)\right](\bar X^\theta_s)ds,\\
\check\Xi(\theta)&:=\int^T_0\left[\left(\nabla_\theta\bar\lambda^\theta\right)^T\cdot\bar q^\theta\cdot\left(\nabla_\theta\bar\lambda^\theta\right)\right](\bar X^\theta_s)ds.
\end{align*}
\end{theorem}

\noindent\textit{Proof.} As before, the arguments for the SMCE follow nearly verbatim those for the MCE. The weaker assumption $\epsilon=o(\Delta)$ suffices in this case because the term for which it was needed in the proof of Theorem \ref{nunormality}, i.e., term IV in (\ref{theabove}), does not appear if one is using the simplified contrast function in place of the original contrast function.

\qed

\begin{remark}
It is clear that one can re-express the conditions $\epsilon=o(\Delta)$ and $\epsilon=o(\Delta^{2})$ in terms of $\delta$ and $\Delta$. We have deliberately chosen instead to present the results in terms of $\epsilon$ and $\Delta$ because $\Delta$ is known in general (via $T$ and $n$) whereas $\epsilon$ can in principle be estimated via the magnitude of the quadratic variation of the process $X^{\varepsilon}$. We emphasize that in contrast to the rest of the subsampling literature, e.g. \cite{azencott2013sub,papavasiliou2009maximum}, we do not (in fact, cannot) impose conditions on how $1/\Delta$ grows relative to $n$, due simply to the fact that we have \textit{fixed} $T=n\cdot\Delta$ to study the regime of small noise.
\end{remark}

\section{Numerical Examples}\label{S:numericalsection}
We now present data from numerical simulations to supplement and illustrate the theory. We begin by considering the system
\begin{align}
dX^{\varepsilon}_t&=\frac\epsilon\delta(\sin(Y^{\varepsilon}_t)-\cos(Y^{\varepsilon}_t))dt-\theta_0X^{\varepsilon}_tdt+\sqrt{\epsilon}dW_t\label{firstsim}\\
Y^{\varepsilon}_t&:=X^{\varepsilon}/\delta\nonumber
\end{align}
for $t\in[0,T=1]$ with $X^{\varepsilon}_0=Y^{\varepsilon}_0=1\in\mathbb{R}$. This system was studied as an example in \cite{spiliopoulos2013maximum}. We suppose that we are in the $\infty$ regime with $\delta = \epsilon^{3/2}$. The results of Sections \ref{S:InferenceFixedNumberObs}, \ref{S:simplified}, and \ref{S:large_n} on consistency and asymptotic normality are illustrated in the numerical data that follow.

The limit $\bar{X}$ of the slow process $X^{\varepsilon}$ in (\ref{firstsim}) is
\begin{align*}
\bar{X}_t&=e^{-t\cdot{\theta_0}\cdot\left(\frac{2\pi}L\right)^2},
\end{align*}
where $L:=\int^{2\pi}_0e^{2(sin(y)+cos(y))}dy$.

We fix $\theta_0=1$ and simulate trajectories using an Euler scheme with $10^8$ evenly-spaced discrete time steps. We simulate both estimators, $\bar\theta^\varepsilon$ and $\tilde\theta^\varepsilon$, $1000$ times for each combination of $\epsilon=10^{-2}, 10^{-3}$ and $n=10, 10^2, 10^3$. Tables \ref{Table1} and \ref{Table2} present in each case the mean estimate, a normal-based confidence interval using the empirical standard deviation, and the theoretical standard deviation as per Theorems \ref{normality} and \ref{simplenormality}. The histograms in Figures \ref{Fig1}-\ref{Fig4} compare the empirical distribution of the estimates with the theoretical density curve.

\begin{table}[h]
\begin{center}
\begin{tabular}{|c|c||c|c|c|c|}
	\hline
	$\epsilon$ & $n$ & Mean Estimator & 68\% Confidence Interval & 95\% Confidence Interval & Theoretical SD\\
	\hline
	\hline
	$10^{-2}$ & $10$ & 1.0173 & (0.5483, 1.4862) & (0.0794, 1.9552) & 0.4370\\ \hline
	$10^{-2}$ & $10^2$ & 1.0430 & (0.5688, 1.5173) & (0.0945, 1.9916) & 0.4370\\ \hline
	$10^{-2}$ & $10^3$ & 1.0686 & (0.6104, 1.5267) & (0.1523, 1.9848) & 0.4370\\ \hline
	$10^{-3}$ & $10$ & 1.0005 & (0.8610, 1.1401) & (0.7215, 1.2794) & 0.1382\\ \hline
	$10^{-3}$ & $10^2$ & 0.9992 & (0.8571, 1.1413) & (0.7150, 1.2834) & 0.1382\\ \hline
	$10^{-3}$ & $10^3$ & 1.0068 & (0.8643, 1.1494) & (0.7218, 1.2919) & 0.1382\\
	\hline
\end{tabular}
\caption{Example 1 - MCE of $\theta_0=1$ with empirical confidence intervals and theoretical standard deviations.}\label{Table1}
\end{center}
\end{table}

\begin{table}[h]
\begin{center}
\begin{tabular}{|c|c||c|c|c|c|}
	\hline
	$\epsilon$ & $n$ & Mean Estimator & 68\% Confidence Interval & 95\% Confidence Interval & Theoretical SD\\
	\hline
	\hline
	$10^{-2}$ & $10$ & 1.0710 & (0.5902, 1.5518) & (0.1095, 2.0326) & 0.4370\\ \hline
	$10^{-2}$ & $10^2$ & 1.0649 & (0.5845, 1.5453) & (0.1041, 2.0256) & 0.4370\\ \hline
	$10^{-2}$ & $10^3$ & 1.0764 & (0.5958, 1.5570) & (0.1151, 2.0377) & 0.4370\\ \hline
	$10^{-3}$ & $10$ & 1.0115 & (0.8732, 1.1498) & (0.7349, 1.2880) & 0.1382\\ \hline
	$10^{-3}$ & $10^2$ & 1.0074 & (0.8598, 1.1549) & (0.7122, 1.3025) & 0.1382\\ \hline
	$10^{-3}$ & $10^3$ & 1.0108 & (0.8690, 1.1525) & (0.7272, 1.2943) & 0.1382\\
	\hline
\end{tabular}
\caption{Example 1 - SMCE of $\theta_0=1$ with empirical confidence intervals and theoretical standard deviations.}\label{Table2}
\end{center}
\end{table}

\begin{figure}[H]
\centering
\begin{minipage}{.45\textwidth}
	\centering
	\includegraphics[width=1\linewidth]{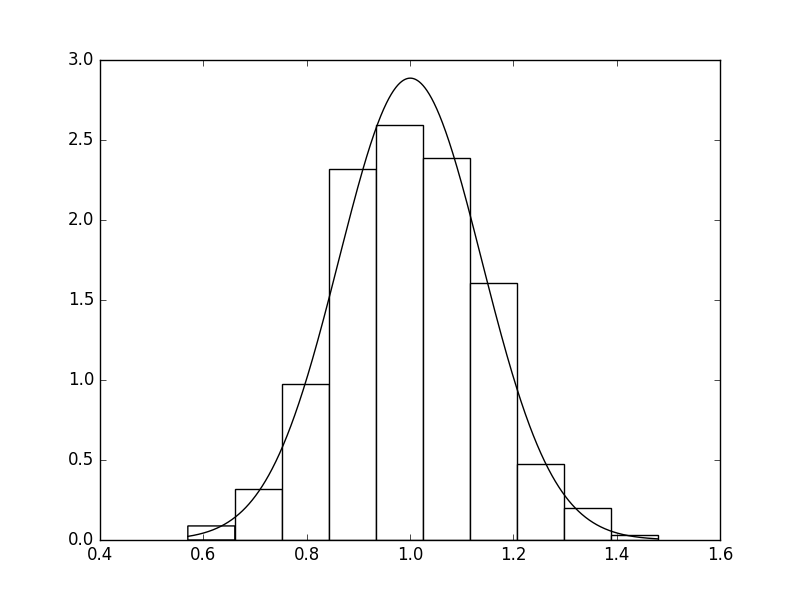}
	\caption{Example 1, MCE, $\epsilon=10^{-3}$, $n=10$}\label{Fig1}
\end{minipage}
\begin{minipage}{.45\textwidth}
	\centering
	\includegraphics[width=1\linewidth]{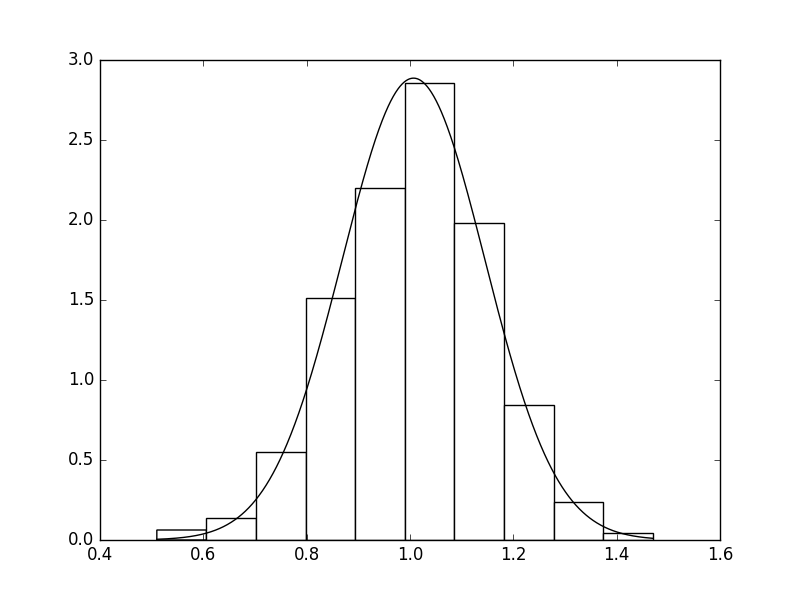}
	\caption{Example 1, MCE, $\epsilon=10^{-3}$, $n=10^3$}\label{Fig2}
\end{minipage}
\end{figure}
\begin{figure}
\begin{minipage}{.45\textwidth}
	\centering
	\includegraphics[width=1\linewidth]{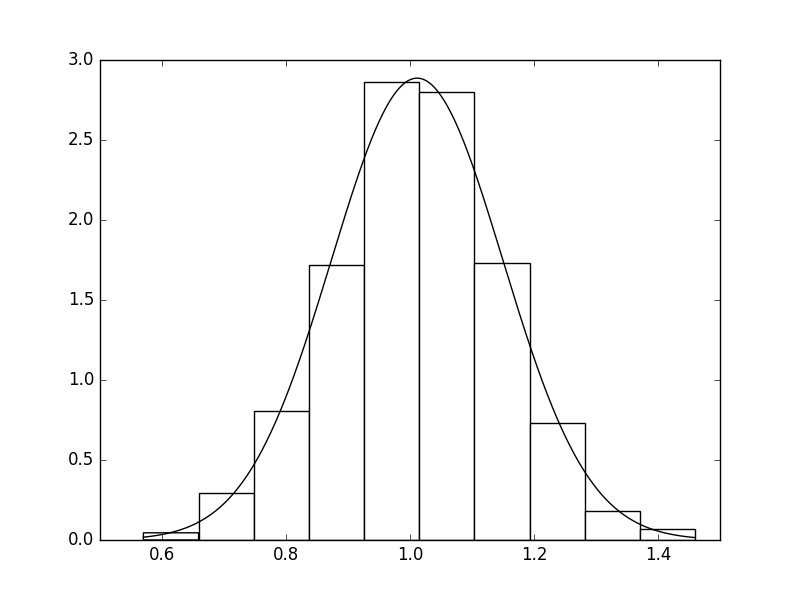}
	\caption{Example 1, SMCE, $\epsilon=10^{-3}$, $n=10$}\label{Fig3}
\end{minipage}
\begin{minipage}{.45\textwidth}
	\centering
	\includegraphics[width=1\linewidth]{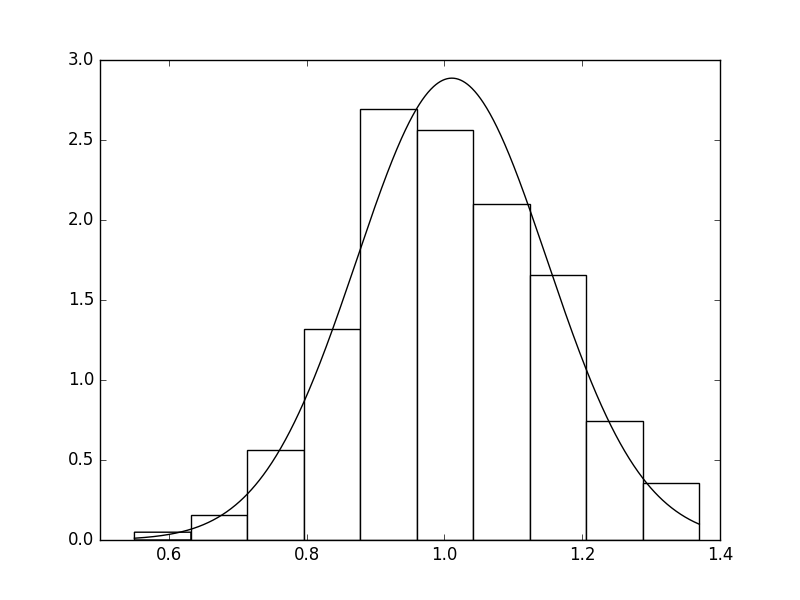}
	\caption{Example 1, SMCE, $\epsilon=10^{-3}$, $n=10^3$}\label{Fig4}
\end{minipage}
\end{figure}

Tables \ref{Table1} and \ref{Table2} and Figures \ref{Fig1}-\ref{Fig4} illustrate that even in the homogenization regime and even with a small number of data points, the minimum contrast estimators are consistent and asymptotically normal. We point out that although the empirical variance of $\tilde\theta^\varepsilon$ was larger than that of $\bar\theta^\varepsilon$, the difference was rather small, especially for small values of $\epsilon$ - this is consistent with the fact that the \textit{theoretical} limiting difference in this example was rather negligible. It is also worth pointing out that coefficients depending on $\theta$ appear in both slow and fast components.

Notice that the dependence on the fast process in (\ref{firstsim}) is periodic and that one may therefore interpret the fast process as taking values in a torus. We now consider a second example to illustrate the case where the fast dynamics are not restricted to a compact space.

\begin{align}
dX^{\varepsilon}_t&=\frac\epsilon\delta\theta_0Y^{\varepsilon}_tdt+\theta_0X^{\varepsilon}_t(Y^{\varepsilon}_t)^2dt+\sqrt{\epsilon}dW_t\label{simmodel}\\
dY^{\varepsilon}_t&=-\frac\epsilon{\delta^2}\frac1{\theta_0}Y^\varepsilon_tdt+\frac{\sqrt\epsilon}{\delta}dB_t\nonumber
\end{align}
for $t\in[0,T=1]$ with $X^{\varepsilon}_0=Y^{\varepsilon}_0=1\in\mathbb{R}$. Again, we suppose that we are in the $\infty$ regime with $\delta = \epsilon^{3/2}$ (this is the more challenging regime, relative to the $\gamma$ regime). As in the first example, coefficients depending on $\theta$ appear in both slow and fast components.

The limit $\bar{X}$ of the slow process $X^{\varepsilon}$ in (\ref{simmodel}) is
\begin{align*}
\bar{X}_t&=e^{t\cdot\frac{\theta_0^2}2}.
\end{align*}

We fix $\theta_0=1$ and simulate trajectories using an Euler scheme with $10^8$ evenly-spaced discrete time steps. We simulate $\bar\theta^\varepsilon$ $1000$ times for each of $\epsilon=10^{-2}, 10^{-3}$ with $n=10$; we simulate $\tilde\theta^\varepsilon$ $1000$ times for each combination of $\epsilon=10^{-2}, 10^{-3}$ and $n=10, 10^2, 10^3$. Tables \ref{Table4} and \ref{Table3} present in each case the mean estimate, a normal-based confidence interval using the empirical standard deviation, and the theoretical standard deviation as per Theorem \ref{normality}. The histograms in Figures \ref{Fig5}-\ref{Fig8} compare the empirical distribution of the estimates with the theoretical density curve.

\begin{table}[h]
\begin{center}
\begin{tabular}{|c|c||c|c|c|c|}
	\hline
	$\epsilon$ & $n$ & Mean Estimator & 68\% Confidence Interval & 95\% Confidence Interval & Theoretical SD\\
	\hline
	\hline
	$10^{-2}$ & $10$ & 1.1042 & (0.9961, 1.2123) & (0.8879, 1.3205) & 0.1079\\ \hline
	$10^{-3}$ & $10$ & 1.0129 & (0.9800, 1.0458) & (0.9472, 1.0787) & 0.0341\\
	\hline
\end{tabular}
\caption{Example 2 - MCE of $\theta_0=1$ with empirical confidence intervals and theoretical standard deviations.}\label{Table4}
\end{center}
\end{table}

\begin{table}[h]
\begin{center}
\begin{tabular}{|c|c||c|c|c|c|}
	\hline
	$\epsilon$ & $n$ & Mean Estimator & 68\% Confidence Interval & 95\% Confidence Interval & Theoretical SD\\
	\hline
	\hline
	$10^{-2}$ & $10$ & 0.9840 & (0.8704, 1.0977) & (0.7567, 1.2114) & 0.1079\\ \hline
	$10^{-2}$ & $10^2$ & 0.9773 & (0.8561, 1.0985) & (0.7349, 1.2197) & 0.1079\\ \hline
	$10^{-2}$ & $10^3$ & 0.9754 & (0.8544, 1.0963) & (0.7335, 1.2173) & 0.1079\\ \hline
	$10^{-3}$ & $10$ & 1.0008 & (0.9662, 1.0354) & (0.9316, 1.0699) & 0.0341\\ \hline
	$10^{-3}$ & $10^2$ & 0.9994 & (0.9654, 1.0333) & (0.9314, 1.0673) & 0.0341\\ \hline
	$10^{-3}$ & $10^3$ & 1.0003 & (0.9659, 1.0347) & (0.9314, 1.0692) & 0.0341\\
	\hline
\end{tabular}
\caption{Example 2 - SMCE of $\theta_0=1$ with empirical confidence intervals and theoretical standard deviations.}\label{Table3}
\end{center}
\end{table}

\begin{figure}[H]
\centering
\begin{minipage}{.45\textwidth}
	\centering
	\includegraphics[width=1\linewidth]{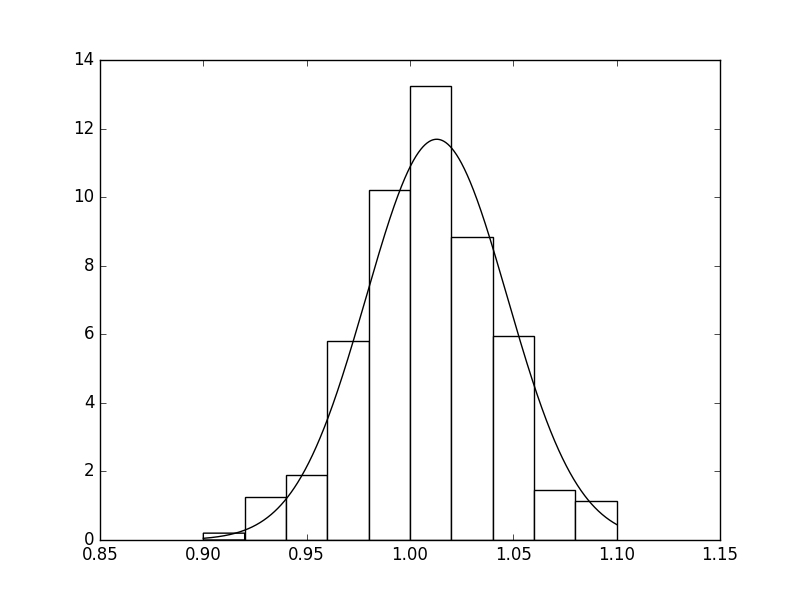}
	\caption{Example 2, MCE, $\epsilon=10^{-3}$, $n=10$}\label{Fig5}
\end{minipage}
\begin{minipage}{.45\textwidth}
	\centering
	\includegraphics[width=1\linewidth]{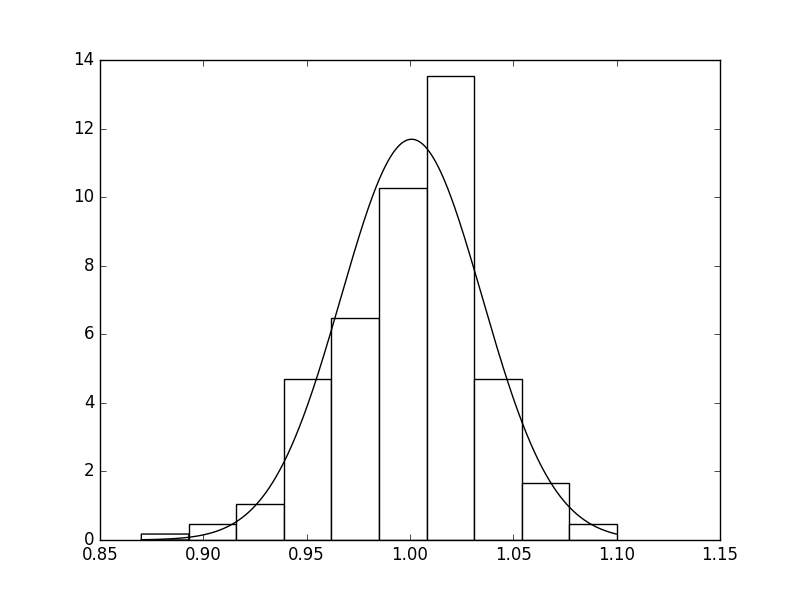}
	\caption{Example 2, SMCE, $\epsilon=10^{-3}$, $n=10$}\label{Fig6}
\end{minipage}
\end{figure}

\begin{figure}[H]
\centering
\begin{minipage}{.45\textwidth}
	\centering
	\includegraphics[width=1\linewidth]{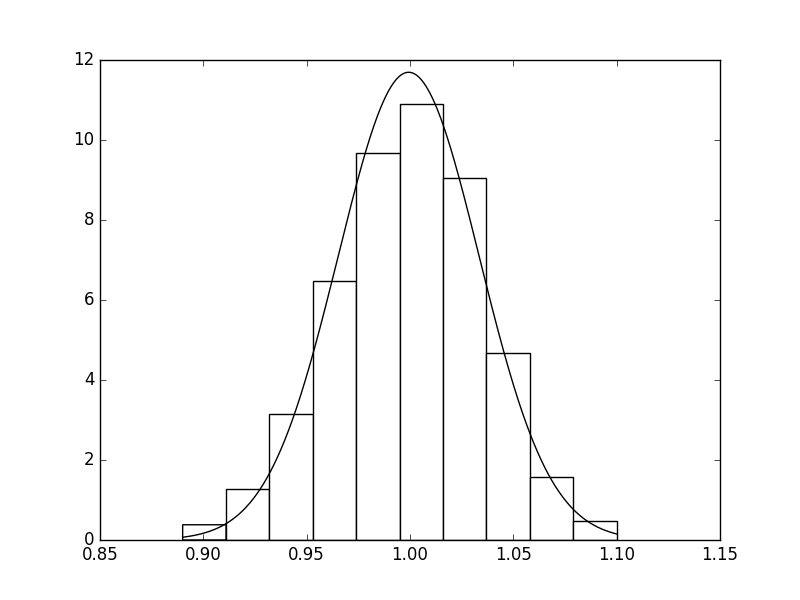}
	\caption{Example 2, SMCE, $\epsilon=10^{-3}$, $n=10^2$}\label{Fig7}
\end{minipage}
\begin{minipage}{.45\textwidth}
	\centering
	\includegraphics[width=1\linewidth]{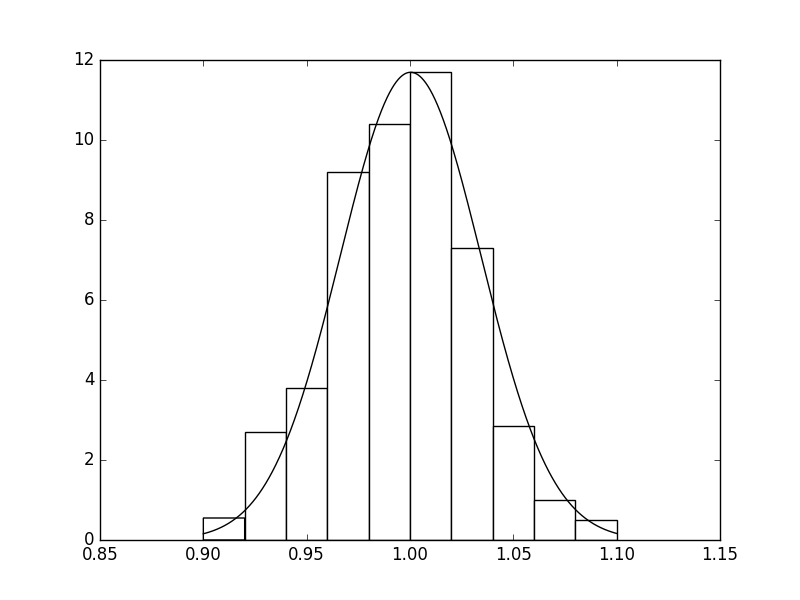}
	\caption{Example 2, SMCE, $\epsilon=10^{-3}$, $n=10^3$}\label{Fig8}
\end{minipage}
\end{figure}

Lastly, we revisit Examples 1 and 2 to see what happens when we deliberately violate the assumption that $\epsilon$ is of lower order than $\Delta:=T/n$. We repeat some of the simulations above with $n=10^6$; the results are presented in Tables \ref{Table5} and \ref{Table6} and Figures \ref{Fig9}-\ref{Fig12}. One sees that the estimators continue to behave well, suggesting that it may be possible to relax our assumptions about the relationship between $\epsilon$ and $\Delta$ in the theory of Section \ref{S:large_n} concerning estimation based on high-frequency observations.
\begin{table}[h]
\begin{center}
\begin{tabular}{|c|c||c|c|c|c|}
	\hline
	Estimator & Mean Estimator & 68\% Confidence Interval & 95\% Confidence Interval & Theoretical SD\\
	\hline
	\hline
	MCE & 1.0427 & (0.5630, 1.5225) & (0.0833, 2.0022) & 0.437\\ \hline
	SMCE & 1.1553 & (0.6839, 1.6268) & (0.2125, 2.0982) & 0.437\\
	\hline
\end{tabular}
\caption{Example 1 revisited - MCE vs SMCE with $\epsilon=10^{-2}$ and large $n$ ($n=10^6$).}\label{Table5}
\end{center}
\end{table}

\begin{figure}[H]
\centering
\begin{minipage}{.45\textwidth}
	\centering
	\includegraphics[width=1\linewidth]{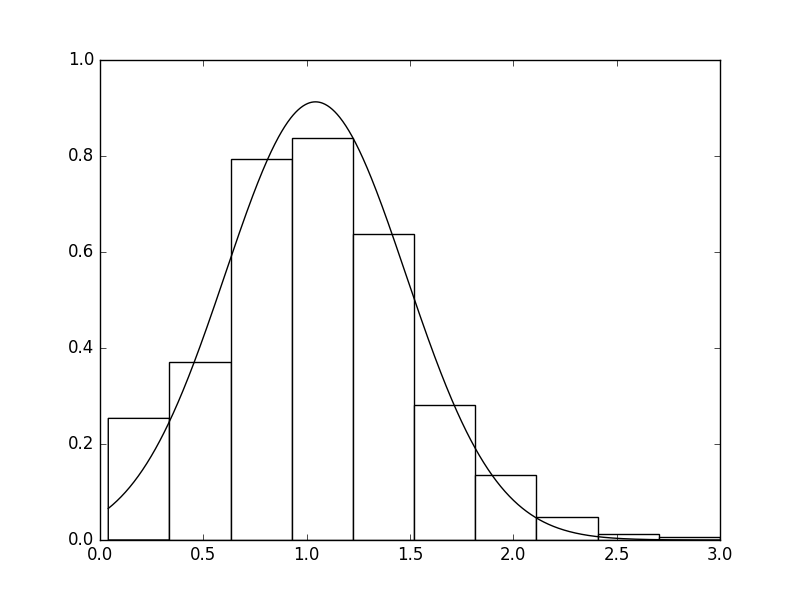}
	\caption{Example 1, MCE, $\epsilon=10^{-2}$, $n=10^6$}\label{Fig9}
\end{minipage}
\begin{minipage}{.45\textwidth}
	\centering
	\includegraphics[width=1\linewidth]{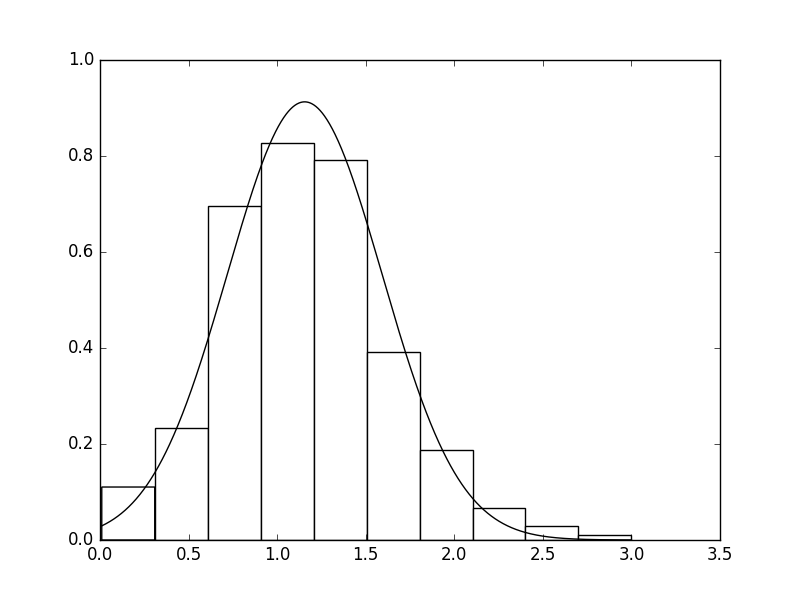}
	\caption{Example 1, SMCE, $\epsilon=10^{-2}$, $n=10^6$}\label{Fig10}
\end{minipage}
\end{figure}

\begin{table}[h]
\begin{center}
\begin{tabular}{|c|c||c|c|c|c|}
	\hline
	$\epsilon$ & Mean Estimator & 68\% Confidence Interval & 95\% Confidence Interval & Theoretical SD\\
	\hline
	\hline
	$10^{-2}$ & 0.9883 & (0.8777, 1.0988) & (0.7671, 1.2094) & 0.1079\\ \hline
	$10^{-3}$ & 1.0005 & (0.9683, 1.0328) & (0.9360, 1.065) & 0.0341\\
	\hline
\end{tabular}
\caption{Example 2 revisited - SMCE of $\theta_0=1$ with large $n$ ($n=10^6$).}\label{Table6}
\end{center}
\end{table}

\begin{figure}[H]
\centering
\begin{minipage}{.45\textwidth}
	\centering
	\includegraphics[width=1\linewidth]{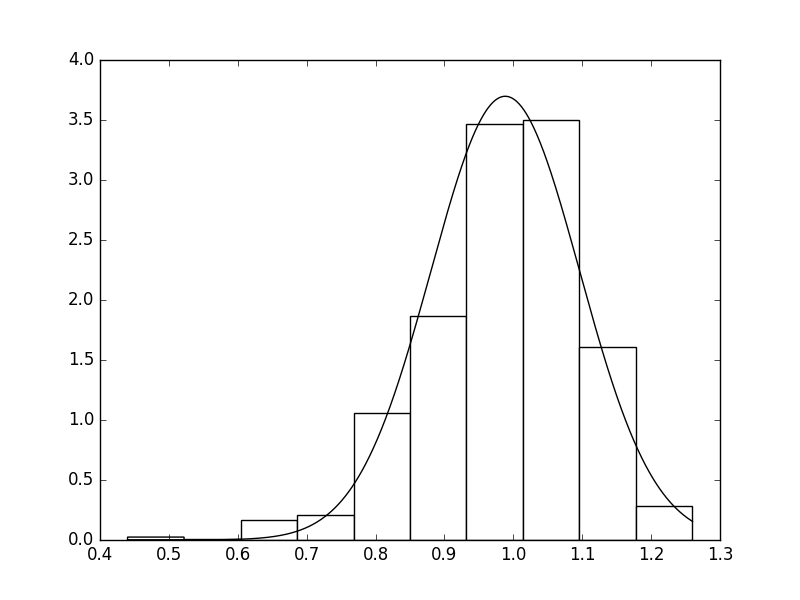}
	\caption{Example 2, SMCE, $\epsilon=10^{-2}$, $n=10^6$}\label{Fig11}
\end{minipage}
\begin{minipage}{.45\textwidth}
	\centering
	\includegraphics[width=1\linewidth]{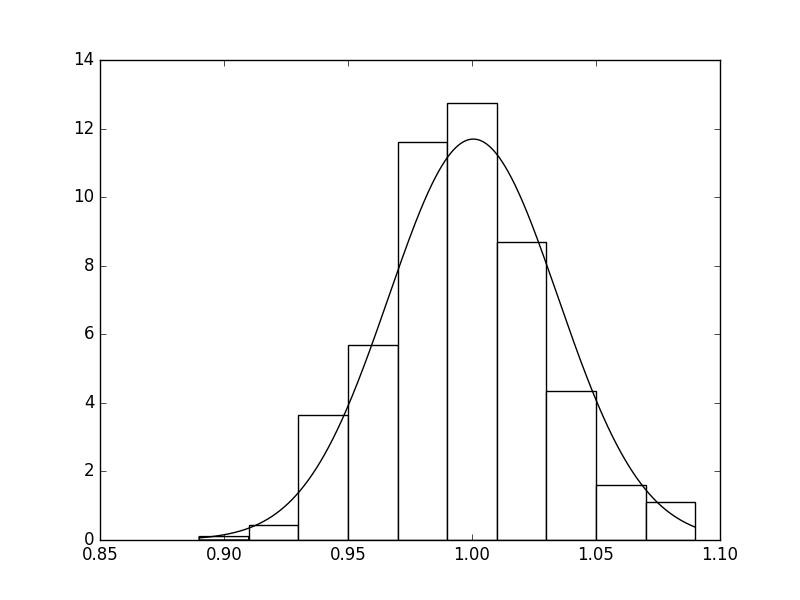}
	\caption{Example 2, SMCE, $\epsilon=10^{-3}$, $n=10^6$}\label{Fig12}
\end{minipage}
\end{figure}

\begin{remark}
Data for the MCE $\bar\theta^\varepsilon$ for the first example (\ref{simmodel}) are presented for $n=10$ only. We have run simulations also with $n=10^2, 10^3$, etc., but the estimators then began to exhibit a strong positive bias. In contrast, a glance at Table \ref{Table3} reveals that the SMCE $\tilde\theta^\varepsilon$ for the very same model produced robust estimates across all combinations of $\epsilon$ and $n$.

 Let us attempt to offer some insight into the reason for this. It turns out that the typical \textit{simplified} contrast curve in this model increases more steeply as one moves rightward from the minimum $\tilde\theta^\varepsilon$ than as one moves leftward. Meanwhile, the effect of introducing the covariance weights $Q^{-1}_k$ in this example is to rescale the contrast by a factor proportional to $\frac{\theta^{2}}{1+\theta^4}\frac{1}{e^{\theta\Delta}-1}$.  This has the potential to balance the leftward and rightward gradients about the new minimum $\bar\theta^\varepsilon$, but at the same time may lead to numerical instabilities when $\Delta$ is small. For a given value of $\epsilon$, the convexity of the simplified contrast about its minimum tends to diminish, whence the introduction of the covariance weights may cause the minimum to move too far to the right.
\end{remark}

\begin{remark}

It should be pointed out that the threshold below which $\varepsilon$ must fall in order for the theorems to apply will depend upon the problem at hand. The fact that SMCE appears both to be more robust than MCE and at the same time easier to compute (one need not compute the weights $Q^{-1}_{k}$) may be viewed as a practical advantage despite the increase in the variance.
\end{remark}

\section{Conclusions}\label{S:Conclusions}
We have presented in this paper discrete-data estimators for unknown parameters in multiscale diffusion models with small noise. The estimators are defined as minimizers of certain contrast functions motivated by a general second-order stochastic Taylor expansion of the slow process; we have shown them to be consistent, asymptotically normal, and, in the case of the MCE, asymptotically statistically efficient in an appropriate sense.

We also considered the case of high-frequency observation, in which the number of samples $n$ increases to infinity concurrently with the vanishing of the small noise. We showed that the asymptotic properties established for the fixed-$n$ case are still valid provided that the sampling interval $\Delta:=T/n$ does not vanish too quickly relative to the small noise. Such conditions are reminiscent of the subsampling prescribed in prior literature on statistical estimation for multiscale models. Despite our best efforts we have not managed to relax them, as has already been done for the case of small noise \textit{without} multiple scales (see \cite{guy2014parametric, SorensenUchida, Uchida}). Nevertheless, the numerical simulations of Section \ref{S:numericalsection} suggest that improvement is possible, posing an important question for future research to investigate. The crux is the development of an optimal characterization of the rate at which the approximation error $\sqrt\epsilon\left[\tilde{\mathcal R}^{\varepsilon,\theta}_{t_k}-Z^{\theta}(t_k,t_{k-1})\tilde{\mathcal R}^{\varepsilon,\theta}_{t_{k-1}}\right]$ vanishes with respect to $\varepsilon=(\epsilon,\delta)$ and $\Delta$; the $\delta$ dependence in particular presents an analytic challenge of unusual delicacy.

Finally, we wish to point out that it seems that the $\mathcal{L}^p$ bounds established in this work can be used to derive Berry-Esseen theorems to characterize the rates of convergence in the central limit theorems.

\section*{Appendix}\label{appendix}
We gather here technical results to which we appeal in the proofs of the main results of this paper.

\subsection*{A.1 Bounds Extended from \cite{gailusspiliopoulos}}\label{SS:ErgodicTheorems}
The SDEs for $X^{\varepsilon}$ and $Y^{\varepsilon}$ that we consider in this paper extend those in \cite{gailusspiliopoulos}. Certain fundamental auxiliary bounds established in \cite{gailusspiliopoulos} are also valid in our case. We gather the necessary bounds in this section. Because the proofs are similar, we do not present them here; the interested reader is referred to \cite{gailusspiliopoulos} for details.

\begin{lemma}\label{moments}
Assume Conditions \ref{basicconditions} and \ref{recurrencecondition} and, in the $\infty$ regime, Condition \ref{centeringcondition}. For any $p>0$ there is a constant $\tilde K$ such that uniformly in $\epsilon$ (and hence also $\delta=\delta(\epsilon))$) sufficiently small,
\begin{align}
E\sup_{0\leq t\leq T}|X^{\varepsilon}_t|^p\leq \tilde K, \nonumber\\
E\sup_{0\leq t\leq T}|X^{\varepsilon}_t-\bar X_t|^p\leq \tilde K.\nonumber
\end{align}
\end{lemma}

This is similar to Lemma 6 in \cite{gailusspiliopoulos}.

\noindent\textit{Sketch of Proof.} The only significant difference relative to Lemma 6 in \cite{gailusspiliopoulos} is the introduction of the asymptotically-singular term $\frac\epsilon\delta b(X^\varepsilon_t,Y^\varepsilon_t)dt$ in the $\infty$ regime. Letting $\chi$ be as in (\ref{chifunction}), applying the It\^o formula to $\chi(X^{\varepsilon}_t,Y^{\varepsilon}_t)$, and rearranging terms, we obtain
\begin{align}
&\int^t_0\frac\epsilon\delta b(X^{\varepsilon}_s,Y^{\varepsilon}_s)ds=\int^t_0\nabla_y\chi\cdot g(X^{\varepsilon}_s,Y^{\varepsilon}_s)ds + \sqrt\epsilon\int^t_0(\nabla_y\chi\cdot\tau_1,\nabla_y\chi\cdot\tau_2)(X^{\varepsilon}_s,Y^{\varepsilon}_s)d(W,B)_s\label{chireference}\\
&\hspace{2pc}+\epsilon\int^t_0\nabla_x\chi\cdot b(X^{\varepsilon}_s,Y^{\varepsilon}_s)ds+\delta\int^t_0\nabla_x\chi\cdot c(X^{\varepsilon}_s,Y^{\varepsilon}_s)ds+\epsilon\delta\int^t_0\sigma\sigma^T:\nabla^2_x\chi(X^{\varepsilon}_s,Y^{\varepsilon}_s)ds\nonumber\\
&\hspace{2pc}+\epsilon\int^t_0\sigma\tau^T_1:\nabla_y\nabla_x\chi(X^{\varepsilon}_s,Y^{\varepsilon}_s)ds+\sqrt\epsilon\delta\int^t_0\nabla_x\chi\cdot\sigma(X^{\varepsilon}_s,Y^{\varepsilon}_s)dW_s-\delta\left(\chi(X^{\varepsilon}_t,Y^{\varepsilon}_t)-\chi(x_0,y_0)\right)\nonumber\\
&=:\int^t_0\nabla_y\chi\cdot g(X^{\varepsilon}_s,Y^{\varepsilon}_s)ds + \mathcal R^\varepsilon_t;\nonumber
\end{align}
hence we can write
\begin{align*}
X^{\varepsilon}_t&=x_0+\int^t_0\lambda_\infty(X^{\varepsilon}_s,Y^{\varepsilon}_s)ds
	+\sqrt{\epsilon}\int^t_0\sigma(X^{\varepsilon}_s,Y^{\varepsilon}_s)dW_s+\mathcal R^\varepsilon_t.\nonumber
\end{align*}

The function $\chi$ and its various derivatives are bounded uniformly in the first variable by polynomials in the second variable (see Chapters 2 and 3 in \cite{gilbargtrudinger}). Noting the vanishing prefactors, we conclude that $\mathcal R^\varepsilon_t$ vanishes. Thus, the proof may proceed as that of Lemma 6 in \cite{gailusspiliopoulos}, with $\lambda_\infty$ taking the place of $c$; we omit the details.

\qed

\begin{lemma}\label{boundlemma}
Assume Conditions \ref{basicconditions} and \ref{recurrencecondition} and, in the $\infty$ regime, Condition \ref{centeringcondition}. Let $h$ be a function of $x$ and $y$ satisfying $|h(x,y)|\leq K(1+|x|^r)(1+|y|^q)$ for some fixed positive constants $K,q,r$, and let $V$ be either of the Wiener processes $W, B$. For any $0<p<\infty$, there is a constant $\tilde K$ such that for $epsilon$ (and hence also $\delta=\delta(\epsilon)$) sufficiently small,
\begin{align*}
E\int^T_0|h(X^\varepsilon_t,Y^\varepsilon_t)|^pdt\leq\tilde K,\\
E\sup_{0\leq t\leq T}\left|\int^t_0h(X^\varepsilon_s,Y^\varepsilon_s)dV_s\right|^p\leq\tilde K.
\end{align*}
\end{lemma}

Given Lemma \ref{moments}, the proof of Lemma \ref{boundlemma} is similar to that of Lemma 7 in \cite{gailusspiliopoulos}; we omit the details.

\begin{lemma}\label{averagelemma} Assume Conditions \ref{basicconditions} and \ref{recurrencecondition} and, in the $\infty$ regime, Condition \ref{centeringcondition}. Let $h(x,y)$ be a function such that $|h(x,y)|\leq K(1+|x|^r)(1+|y|^q)$ for some fixed positive constants $K,q,r$, and which, along with each of its derivatives up to second order, is H\"older continuous in $y$ uniformly in $x$ with absolute value growing at most polynomially in $|y|$ as $y\to\infty$. For any $0<p<\infty$, there is a constant $\tilde K$ such that for $\epsilon$ (and hence also $\delta=\delta(\epsilon)$) sufficiently small,
\begin{align}
E\sup_{0\leq t\leq T}\bigg|\int^t_0\Big(h(X^\varepsilon_s,Y^\varepsilon_s)
-\bar h(X^{\varepsilon}_s)\Big)ds\bigg|^p\leq \tilde K(\sqrt\epsilon+\sqrt\delta)^p,\nonumber
\end{align}
where $\bar h(x)$ is the averaged function $\int_\mathcal{Y}h(x,y)\mu_x(dy)$.
\end{lemma}

Given Lemma \ref{moments}, the proof of Lemma \ref{averagelemma} is similar to that of Theorem 4 in \cite{gailusspiliopoulos}; we omit the details.

\begin{lemma}\label{ergodiclemma} Assume Conditions \ref{basicconditions} and \ref{recurrencecondition} and, in the $\infty$ regime, Condition \ref{centeringcondition}. Let $h(x,y)$ be a function such that $|h(x,y)|\leq K(1+|x|^r)(1+|y|^q)$ for some fixed positive constants $K,q,r$, and which, along with each of its derivatives up to second order, is H\"older continuous in $y$ uniformly in $x$ with absolute value growing at most polynomially in $|y|$ as $y\to\infty$. For any $0<p<\infty$, there is a constant $\tilde K$ such that for $\epsilon$ (and hence also $\delta=\delta(\epsilon)$) sufficiently small,
\begin{align}
E\sup_{0\leq t\leq T}\bigg|\int^t_0\Big(h(X^\varepsilon_s,Y^\varepsilon_s)
-\bar h(\bar{X}_s)\Big)ds\bigg|^p\leq \tilde K(\sqrt\epsilon+\sqrt\delta)^p,\nonumber
\end{align}
where $\bar h(x)$ is the averaged function $\int_\mathcal{Y}h(x,y)\mu_x(dy)$.
\end{lemma}

Given Theorem \ref{xlimit} and Lemma \ref{averagelemma}, the proof of Lemma \ref{ergodiclemma} is similar to that of Lemma 10 in \cite{gailusspiliopoulos}; we omit the details.

\subsection*{A.2 Lemmata on F and Q}
\begin{lemma}\label{fbound} Assume Conditions \ref{basicconditions} and \ref{recurrencecondition} and, in the $\infty$ regime, Condition \ref{centeringcondition}. There is a constant $\tilde K$ such that for $\epsilon$ (and hence also $\delta=\delta(\epsilon))$ sufficiently small,
\begin{align}
E|F^\varepsilon_k(\theta;\{X^{\epsilon, \theta}_{t_k}\}^n_{k=1})|^p\leq\tilde K\cdot\epsilon^{p/2},\label{firstfbound}\\
E|\tilde F_k(\theta;\{X^{\epsilon, \theta}_{t_k}\}^n_{k=1})|^p\leq\tilde K\cdot\epsilon^{p/2}\label{secondfbound}.
\end{align}
\end{lemma}

\noindent\textit{Proof.} Recalling that
\begin{align}
F^\varepsilon_k(\theta; \{X^{\epsilon, \theta}_{t_k}\}_{k=1}^n)&=\left[[X^{\epsilon, \theta}_{t_k}-\bar X^\theta_{{t_k}}]-Z^{\theta}(t_k,t_{k-1})\cdot[X^{\epsilon, \theta}_{t_{k-1}}-\bar X^\theta_{{t_{k-1}}}]\right]-\sqrt\epsilon\int^{t_k}_{t_{k-1}}Z^{\theta}(t_k,s)\cdot\bar  J^\theta(\bar X^\theta_{s})ds,\label{fboundreference}
\end{align}
The first statement, (\ref{firstfbound}), follows by the triangle inequality and Theorem \ref{xlimit}.

(\ref{secondfbound}) may be obtained in the same way by omitting the last term in (\ref{fboundreference}).

\qed

In addition to Lemma \ref{fbound}, we need a lemma to establish that $E|F^\varepsilon_k(\theta_2;\{X^{\epsilon, \theta_1}_{t_k}\}^n_{k=1})|$ and $E|\tilde F_k(\theta_2;\{X^{\epsilon, \theta_1}_{t_k}\}^n_{k=1})|$ are uniformly bounded in $\epsilon$ sufficiently small \textit{even as $\theta_1$ and $\theta_2$ vary independently}. Of course, with two distinct values of $\theta$, one should not expect the functions to vanish in the limit.

\begin{lemma}\label{flemma} Assume Conditions \ref{basicconditions} and \ref{recurrencecondition} and, in the $\infty$ regime, Condition \ref{centeringcondition}. There is a constant $\tilde K$ such that for $\epsilon$ (and hence also $\delta=\delta(\epsilon))$ sufficiently small,
\begin{align}
\sup_{(\theta_1, \theta_2)\in\bar\Theta^2}E|F^\varepsilon_k(\theta_2;\{X^{\epsilon, \theta_1}_{t_k}\}^n_{k=1})|\leq\tilde K,\label{firstflemma}\\
\sup_{(\theta_1, \theta_2)\in\bar\Theta^2}E|\tilde F_k(\theta_2;\{X^{\epsilon, \theta_1}_{t_k}\}^n_{k=1})|\leq\tilde K.\label{secondflemma}
\end{align}
\end{lemma}

\noindent\textit{Proof.} Notice that
\begin{align}
\sup_{(\theta_1, \theta_2)\in\bar\Theta^2}|F^\varepsilon_k(\theta_2;\{X^{\epsilon, \theta_1}_{t_k}\}^n_{k=1})|\leq I+II+III+IV+V,\label{flemmareference}
\end{align}
where
\begin{align}
I&:=\sup_{\theta\in\bar\Theta}|X^{\varepsilon,\theta}_{t_k}-X^{\varepsilon,\theta}_{t_{k-1}}|;\label{supxincrement}
\end{align}
\begin{align*}
II&:=\sup_{\theta\in\bar\Theta}|\bar X^{\theta}_{t_k}-\bar X^{\theta}_{t_{k-1}}|\leq\left[\sup_{\theta\in\bar\Theta, 0\leq t\leq T}|(\nabla_x\bar\lambda^\theta)(\bar X^\theta_t)|\right]\cdot\Delta;
\end{align*}
\begin{align*}
III&:=\sup_{\theta\in\bar\Theta}\left|1_{m}-Z^{\theta}(t_k,t_{k-1})\right|\cdot\sup_{\theta\in\bar\Theta}|X^{\varepsilon,\theta}_{t_{k-1}}|;\\
&\leq\left[\sup_{\theta\in\bar\Theta, 0\leq t\leq T}|Z^{\theta}(t,0)\cdot(\nabla_x\bar\lambda^\theta)(\bar X^\theta_t)|\right]\cdot\left[\sup_{\theta\in\bar\Theta, 0\leq t\leq T}|X^{\varepsilon,\theta}_t|\right]\cdot\Delta;
\end{align*}
\begin{align*}
IV&:=\sup_{\theta\in\bar\Theta}\left|1_{m}-Z^{\theta}(t_k,t_{k-1})\right|\cdot\sup_{\theta\in\bar\Theta}|\bar X^{\theta}_{t_{k-1}}|\\
&\leq\left[\sup_{\theta\in\bar\Theta, 0\leq t\leq T}|Z^{\theta}(t,0)\cdot(\nabla_x\bar\lambda^\theta)(\bar X^\theta_t)|\right]\cdot\left[\sup_{\theta\in\bar\Theta, 0\leq t\leq T}|\bar X^{\theta}_t|\right]\cdot\Delta;
\end{align*}
\begin{align*}
V&:=\sup_{\theta\in\bar\Theta}\left|\sqrt\epsilon\int^{t_k}_{t_{k-1}}Z^{\theta}(t_k,s)\cdot\bar J^\theta(\bar X^\theta_s)ds\right|\leq\sqrt\epsilon\cdot\left[\sup_{\theta\in\bar\Theta, 0\leq t\leq T}|Z^{\theta}(t,0)\cdot\bar J^\theta(\bar X^\theta_t)|\right]\cdot\Delta.
\end{align*}

By Lemma \ref{moments}, there is a constant $K$ such that for $\epsilon$ sufficiently small, $EI\leq K$. Similarly, there is a (perhaps larger) constant $K$ such that for $\epsilon$ sufficiently small, $II+EIII+IV+V\leq K\Delta$. The first statement, (\ref{firstflemma}), follows easily.

(\ref{secondflemma}) may be obtained in the same way by omitting the last term in (\ref{flemmareference}).

\qed

\begin{lemma}\label{flemmabar} Assume Conditions \ref{basicconditions} and \ref{recurrencecondition} and, in the $\infty$ regime, Condition \ref{centeringcondition}. There is a constant $\tilde K$ such that for all $n>0$, $k=1, 2, $ \ldots $n$, and $\epsilon$ (and hence also $\delta=\delta(\epsilon)$) sufficiently small,
\begin{align}
\sup_{(\theta_1, \theta_2)\in\bar\Theta^2}|F^\varepsilon_k(\theta_2;\{\bar X^{\theta_1}_{t_k}\}^n_{k=1})|& \leq\Delta\cdot\tilde K,\label{firstflemmabar}\\
\sup_{(\theta_1, \theta_2)\in\bar\Theta^2}|\tilde F^\varepsilon_k(\theta_2;\{\bar X^{\theta_1}_{t_k}\}^n_{k=1})|& \leq\Delta\cdot\tilde K.\label{secondflemmabar}
\end{align}
\end{lemma}

\noindent\textit{Proof.} We write
\begin{align}
\sup_{(\theta_1, \theta_2)\in\bar\Theta^2}|F^\varepsilon_k(\theta_2;\{\bar X^{\theta_1}_{t_k}\}^n_{k=1})|\leq 2\cdot I+2\cdot II+III,\label{flemmabarreference}
\end{align}
where
\begin{align*}
I&:=\sup_{\theta\in\bar\Theta}|\bar X^{\theta}_{t_k}-\bar X^{\theta}_{t_{k-1}}|\leq\left[\sup_{\theta\in\bar\Theta, 0\leq t\leq T}|(\nabla_x\bar\lambda^\theta)(\bar X^\theta_t)|\right]\cdot\Delta;
\end{align*}
\begin{align*}
II&:=\sup_{\theta\in\bar\Theta}\left|1_{m}-Z^{\theta}(t_k,t_{k-1})\right|\cdot\sup_{\theta\in\bar\Theta}|\bar X^{\theta}_{t_{k-1}}|\leq\left[\sup_{\theta\in\bar\Theta, 0\leq t\leq T}|Z^{\theta}(t,0)\cdot(\nabla_x\bar\lambda^\theta)(\bar X^\theta_t)|\right]\cdot\left[\sup_{\theta\in\bar\Theta, 0\leq t\leq T}|\bar X^{\theta}_t|\right]\cdot\Delta;
\end{align*}
\begin{align*}
III&:=\sup_{\theta\in\bar\Theta}\left|\sqrt\epsilon\int^{t_k}_{t_{k-1}}Z^{\theta}(t_k,s)\cdot\bar J^\theta(\bar X^\theta_s)ds\right|\leq\sqrt\epsilon\cdot\left[\sup_{\theta\in\bar\Theta, 0\leq t\leq T}|Z^{\theta}(t,0)\cdot\bar J^\theta(\bar X^\theta_t)|\right]\cdot\Delta.
\end{align*}

The first statement, (\ref{firstflemmabar}), follows easily. (\ref{secondflemmabar}) may be obtained in the same way by omitting the last term in (\ref{flemmabarreference}).
\qed

\begin{lemma}\label{qlemma} Assume Conditions \ref{basicconditions} and \ref{recurrencecondition} and, in the $\infty$ regime, Condition \ref{centeringcondition}. There is a positive constant $\tilde K$ such that for all $n>0$ and $k=1, 2, $ \ldots $n$,
\begin{align*}
\inf_{\theta\in\Theta}||Q_k(\theta)||\geq\Delta\cdot\tilde K;
\end{align*}
in particular, the inverse matrices $Q^{-1}_k$ exist and $\sup_{\theta\in\Theta}||Q^{-1}_k(\theta)||\leq\Delta^{-1}\cdot\tilde K^{-1}$.
\end{lemma}

\noindent\textit{Proof.} Recalling our nondegeneracy assumptions, $\bar q^\theta$ is uniformly nondegenerate by the argument of Theorem 11.3 in \cite{pavliotis2008multiscale}. The result follows readily upon noting that the norms of the exponential terms in the definition of $Q$ are bounded uniformly in all parameters away from zero.

\qed

\subsection*{A.3 Lemmata on the Contrast Function}

\begin{lemma}(Limit of the Contrast Function)\label{contrastlimit}
Assume Conditions \ref{basicconditions} and \ref{recurrencecondition} and, in the $\infty$ regime, Condition \ref{centeringcondition}. Fix $n>0$. For any $\eta>0$,
\begin{align*}
\lim_{\epsilon\to0}P\left(\sup_{\theta_1,\theta_2\in\Theta}|U^\varepsilon(\theta_2; \{X^{\epsilon,\theta_1}_{t_k}\}_{k=1}^n)-\tilde U(\theta_2; \{\bar X^{\theta_1}_{t_k}\}_{k=1}^n)|>\eta\right)=0.
\end{align*}
\end{lemma}

\noindent\textit{Proof.} We begin by writing
\begin{align*}
U^\varepsilon(\theta_2; \{X^{\epsilon,\theta_1}_{t_k}\}_{k=1}^n)-\tilde U(\theta_2; \{\bar X^{\theta_1}_{t_k}\}_{k=1}^n)&=\left[U^\varepsilon(\theta_2; \{X^{\epsilon,\theta_1}_{t_k}\}_{k=1}^n)-U^\varepsilon(\theta_2; \{\bar X^{\theta_1}_{t_k}\}_{k=1}^n)\right]\\
&\hspace{2pc}+\left[U^\varepsilon(\theta_2; \{\bar X^{\theta_1}_{t_k}\}_{k=1}^n)-\tilde U(\theta_2; \{\bar X^{\theta_1}_{t_k}\}_{k=1}^n)\right]\nonumber\\
&=\left[\Sigma_{k=1}^nA^\varepsilon_k(\theta_1,\theta_2)\right]+\left[\Sigma_{k=1}^nB^\varepsilon_k(\theta_1,\theta_2)\right],
\end{align*}
where
\begin{align*}
A^\varepsilon_k(\theta_1, \theta_2)&:=\left[F^\varepsilon_k(\theta_2;\{X^{\epsilon,\theta_1}_{t_k}\}_{k=1}^n)-F^\varepsilon_k(\theta_2;\{\bar X^{\theta_1}_{t_k}\}_{k=1}^n)\right]^TQ_k^{-1}(\theta_2)\left[F^\varepsilon_k(\theta_2;\{X^{\epsilon,\theta_1}_{t_k}\}_{k=1}^n)+F^\varepsilon_k(\theta_2;\{\bar X^{\theta_1}_{t_k}\}_{k=1}^n)\right],\nonumber\\
B^\varepsilon_k(\theta_1, \theta_2)&:=\left[F^\varepsilon_k(\theta_2;\{\bar X^{\theta_1}_{t_k}\}_{k=1}^n)-\tilde F_k(\theta_2;\{\bar X^{\theta_1}_{t_k}\}_{k=1}^n)\right]^TQ_k^{-1}(\theta_2)\left[F^\varepsilon_k(\theta_2;\{\bar X^{\theta_1}_{t_k}\}_{k=1}^n)+\tilde F_k(\theta_2;\{\bar X^{\theta_1}_{t_k}\}_{k=1}^n)\right].\nonumber
\end{align*}
As we have taken $n$ to be fixed, it suffices to show that each term vanishes in probability uniformly in $\theta_1, \theta_2$ as $\epsilon\to0$.

For $A^\varepsilon_k$, we have that
\begin{align}
Q_k^{-1}(\theta_2)\left[F^\varepsilon_k(\theta_2;\{X^{\epsilon,\theta_1}_{t_k}\}_{k=1}^n)+F^\varepsilon_k(\theta_2;\{\bar X^{\theta_1}_{t_k}\}_{k=1}^n)\right]\label{arh}
\end{align}
is bounded by Lemmata \ref{flemma}, \ref{flemmabar}, and \ref{qlemma} while
\begin{align}
F^\varepsilon_k(\theta_2;\{X^{\epsilon,\theta_1}_{t_k}\}_{k=1}^n)-F^\varepsilon_k(\theta_2;\{\bar X^{\theta_1}_{t_k}\}_{k=1}^n)&=[X^{\epsilon,\theta_1}_{t_k}-\bar X^{\theta_1}_{t_k}]-Z^{\theta_{2}}(t_k,t_{k-1})\cdot[X^{\epsilon,\theta_1}_{t_{k-1}}-\bar X^{\theta_1}_{t_{k-1}}]\label{alh}
\end{align}
vanishes by Theorem \ref{xlimit} and the fact that the Frobenius norm of $Z^{\theta_{2}}(t_k,t_{k-1})$ is bounded.

For $B^\varepsilon_k$, we have that
\begin{align}
Q_k^{-1}(\theta_2)\left[F^\varepsilon_k(\theta_2;\{\bar X^{\theta_1}_{t_k}\}_{k=1}^n)+\tilde F_k(\theta_2;\{\bar X^{\theta_1}_{t_k}\}_{k=1}^n)\right]\label{brh}
\end{align}
is bounded by Lemmata \ref{flemmabar} and \ref{qlemma} while
\begin{align}
F^\varepsilon_k(\theta_2;\{\bar X^{\theta_1}_{t_k}\}_{k=1}^n)-\tilde F_k(\theta_2;\{\bar X^{\theta_1}_{t_k}\}_{k=1}^n)&=-\sqrt\epsilon\int^{t_k}_{t_{k-1}}Z^{\theta_{2}}(t_k,s)\cdot\bar J^{\theta_2}(\bar X^{\theta_2}_s)ds\label{blh}
\end{align}
vanishes by the vanishing prefactor $-\sqrt\epsilon$ and the fact that the Euclidean norm of  $\int_{t_{k-1}}^{t_{k}}Z^{\theta_{2}}(t_k,s)\cdot\bar J^{\theta_2}(\bar X^{\theta_2}_s)ds$ is bounded (indeed, proportionally to $\Delta$).

\qed

\begin{lemma}\label{modulus}
Assume Conditions \ref{basicconditions} and \ref{recurrencecondition} and, in the $\infty$ regime, Condition \ref{centeringcondition}. There is a constant $\tilde K$ such that for any $\theta\in\Theta$ and $n>0$,
\begin{align*}
\sup_{\theta\in\Theta}\tilde w^\theta(\phi)\leq \tilde K\cdot\phi,
\end{align*}
where
\begin{align*}
\tilde w^\theta(\phi):=\sup_{\theta_1,\theta_2\in\Theta; |\theta_1-\theta_2|\leq\phi}\left|\tilde U(\theta_1; \{\bar X^{\theta}_{t_k}\}_{k=1}^n)-\tilde U(\theta_2; \{\bar X^{\theta}_{t_k}\}_{k=1}^n)\right|.
\end{align*}
\end{lemma}

\noindent\textit{Proof.}
We write
\begin{align*}
\tilde U(\theta_1; \{\bar X^{\theta}_{t_k}\}_{k=1}^n)-\tilde U(\theta_2; \{\bar X^{\theta}_{t_k}\}_{k=1}^n)&=\left[\Sigma_{k=1}^nC_k(\theta_1,\theta_2)\right]+\left[\Sigma_{k=1}^nD_k(\theta_1,\theta_2)\right],
\end{align*}
where
\begin{align}
C_k(\theta_1, \theta_2)&:=\left[\tilde F_k(\theta_1;\{\bar X^{\theta}_{t_k}\}_{k=1}^n)-\tilde F_k(\theta_2;\{\bar X^{\theta}_{t_k}\}_{k=1}^n)\right]^TQ_k^{-1}(\theta_1)\left[\tilde F_k(\theta_1;\{\bar X^{\theta}_{t_k}\}_{k=1}^n)+\tilde F_k(\theta_2;\{\bar X^{\theta}_{t_k}\}_{k=1}^n)\right],\nonumber\\
D_k(\theta_1, \theta_2)&:=\tilde F^T_k(\theta_2;\{\bar X^{\theta}_{t_k}\}_{k=1}^n)\cdot\left[Q_k^{-1}(\theta_1)-Q_k^{-1}(\theta_2)\right]\cdot\tilde F_k(\theta_2;\{\bar X^{\theta}_{t_k}\}_{k=1}^n)\nonumber\\
&=\tilde F^T_k(\theta_2;\{\bar X^{\theta}_{t_k}\}_{k=1}^n)\cdot Q_k^{-1}(\theta_1)\cdot\left[Q_k(\theta_2)-Q_k(\theta_1)\right]\cdot Q_k^{-1}(\theta_2)\cdot\tilde F_k(\theta_2;\{\bar X^{\theta}_{t_k}\}_{k=1}^n).\nonumber
\end{align}
It suffices to show that there is a constant $K$ such that the absolute value of each term is bounded by $K\cdot\Delta\cdot|\theta_1-\theta_2|$.

For $C_k$, the terms $Q_k^{-1}(\theta_1)\left[\tilde F_k(\theta_1;\{\bar X^{\theta}_{t_k}\}_{k=1}^n)+\tilde F_k(\theta_2;\{\bar X^{\theta}_{t_k}\}_{k=1}^n)\right]$ are bounded by Lemmata \ref{flemmabar} and \ref{qlemma}, so it will suffice to bound the terms $\left|\tilde F_k(\theta_1;\{\bar X^{\theta}_{t_k}\}_{k=1}^n)-\tilde F_k(\theta_2;\{\bar X^{\theta}_{t_k}\}_{k=1}^n)\right|$, individually, by a certain constant times $\Delta\cdot|\theta_1-\theta_2|$.

Notice that
\begin{align}
\tilde F_k(\theta_1;\{\bar X^{\theta}_{t_k}\}_{k=1}^n)-\tilde F_k(\theta_2;\{\bar X^{\theta}_{t_k}\}_{k=1}^n)&=I+II+III,\nonumber
\end{align}
where
\begin{align}
I&:=[\bar X^{\theta_2}_{t_k}-\bar X^{\theta_1}_{t_k}]-[\bar X^{\theta_2}_{t_{k-1}}-\bar X^{\theta_1}_{t_{k-1}}],\nonumber\\
II&:=\left[Z^{\theta_{1}}(t_k,t_{k-1})-1_{m}\right]\cdot\left[\bar X^{\theta_1}_{t_{k-1}}-\bar X^{\theta_2}_{t_{k-1}}\right],\nonumber\\
III&:=\left[Z^{\theta_{1}}(t_k,t_{k-1})-Z^{\theta_{2}}(t_k,t_{k-1})\right]\cdot\left[\bar X^{\theta_2}_{t_{k-1}}-\bar X^\theta_{t_{k-1}}\right].\nonumber
\end{align}

By the triangle inequality,
\begin{align}
\left|\tilde F_k(\theta_1;\{\bar X^{\theta}_{t_k}\}_{k=1}^n)-\tilde F_k(\theta_2;\{\bar X^{\theta}_{t_k}\}_{k=1}^n)\right|&\leq\left|I\right|+\left|II\right|+\left|III\right|.\nonumber
\end{align}
Firstly,
\begin{align}
\left|I\right|&\leq\int^{t_k}_{t_{k-1}}|\bar\lambda^{\theta_2}(\bar X^{\theta_2}_s)-\bar\lambda^{\theta_1}(\bar X^{\theta_1}_s)|ds\leq\left[\sup_{\tilde\theta\in\Theta, 0\leq t\leq T}|\nabla_\theta(\bar\lambda^{\tilde\theta}(X^{\tilde\theta}_t))|\right]\cdot\Delta\cdot|\theta_1-\theta_2|;\nonumber
\end{align}
secondly,
\begin{align}
\left|II\right|&\leq \left[\sup_{\tilde\theta\in\Theta,0\leq t\leq T}|(\nabla_x\bar\lambda^{\tilde\theta})(\bar X^{\tilde\theta}_t)|\right]\cdot\Delta\cdot \int^{t_{k-1}}_0|\bar\lambda^{\theta_1}(\bar X^{\theta_1}_s)-\bar\lambda^{\theta_2}(\bar X^{\theta_2}_s)|ds\nonumber\\
&\leq\left[\sup_{\tilde\theta\in\Theta,0\leq t\leq T}|(\nabla_x\bar\lambda^{\tilde\theta})(\bar X^{\tilde\theta}_t)|\right]\cdot\left[\sup_{\tilde\theta\in\Theta, 0\leq t\leq T}|\nabla_\theta(\bar\lambda^{\tilde\theta}(X^{\tilde\theta}_t))|\right]\cdot T\cdot\Delta\cdot|\theta_1-\theta_2|;\nonumber
\end{align}
thirdly,
\begin{align}
\left|III\right|&\leq\left|Z^{\theta_{1}}(t_k,t_{k-1})\right|\cdot\left|\int^{t_k}_{t_{k-1}}(\nabla_x\bar\lambda^{\theta_2})(\bar X^{\theta_2}_u)-(\nabla_x\bar\lambda^{\theta_1})(\bar X^{\theta_1}_u)du\right|\cdot \left|\bar X^{\theta_2}_{t_{k-1}}-\bar X^\theta_{t_{k-1}}\right|\nonumber\\
&\leq\left[\sup_{\tilde\theta\in\Theta,0\leq t\leq T}|Z^{\theta}(t,0)|\right]\cdot \left[\sup_{\tilde\theta\in\Theta,0\leq t\leq T}|\nabla_\theta\left((\nabla_x\bar\lambda^{\tilde\theta})(\bar X^{\tilde\theta}_t)\right)|\right]\cdot\left[\sup_{\tilde\theta\in\Theta, 0\leq t\leq T}2|\bar X^{\tilde\theta}_t|\right]\cdot\Delta\cdot|\theta_1-\theta_2|;\nonumber
\end{align}

hence indeed there is a constant $K$ such that
\begin{align}
\left|\tilde F_k(\theta_1;\{\bar X^{\theta}_{t_k}\}_{k=1}^n)-\tilde F_k(\theta_2;\{\bar X^{\theta}_{t_k}\}_{k=1}^n)\right|&\leq K\cdot\Delta\cdot|\theta_1-\theta_2|.\nonumber
\end{align}

For $D_k$, the terms $Q_k^{-1}(\theta_i)\cdot\tilde F_k(\theta_2;\{\bar X^{\theta}_{t_k}\}_{k=1}^n)$ are bounded by Lemmata \ref{flemmabar} and \ref{qlemma}, so using the fact that the Frobenius norm is equivalent to the operator norm, it will suffice to bound the terms $\left|Q_k(\theta_2)-Q_k(\theta_1)\right|$, individually, by a certain constant times $\Delta\cdot|\theta_1-\theta_2|$.

We have
\begin{align*}
Q_k(\theta_2)-Q_k(\theta_1)&=IV+V,
\end{align*}
where
\begin{align*}
IV&:=\int^{t_k}_{t_{k-1}}Z^{\theta_{2}}(t_k,s)\cdot \left(\bar q^{\theta_2}(\bar X^{\theta_2}_{s})-\bar q^{\theta_1}(\bar X^{\theta_1}_{s})\right)\cdot e^{\int^{t_k}_s(\nabla_x\bar\lambda^{\theta_2})^T(\bar X^{\theta_2}_{u})du}ds,\\
V&:=\int^{t_k}_{t_{k-1}} \left(Z^{\theta_{2}}(t_k,s)-Z^{\theta_{1}}(t_k,s)\right)\cdot \bar q^{\theta_1}(\bar X^{\theta_1}_{s})\cdot \left(Z^{\theta_{2}}(t_k,s)+Z^{\theta_{1}}(t_k,s)\right);
\end{align*}
by the triangle inequality,
\begin{align*}
\left|Q_k(\theta_2)-Q_k(\theta_1)\right|&\leq\left|IV\right|+\left|V\right|.
\end{align*}

Firstly,
\begin{align*}
\left|IV\right|&\leq\left[\sup_{\tilde\theta\in\Theta,0\leq t\leq T}\left|Z^{\tilde\theta}(t,0)(Z^{\tilde\theta})^{T}(t,0)\right|\right]\cdot\left[\sup_{\tilde\theta\in\Theta,0\leq t\leq T}||\nabla_\theta\left(\bar q^{\tilde\theta}(\bar X^{\tilde\theta}_t)\right)||\right]\cdot\Delta\cdot|\theta_1-\theta_2|;
\end{align*}
secondly,
\begin{align*}
\left|V\right|&\leq\left[\sup_{\tilde\theta\in\Theta,0\leq t\leq T}|Z^{\theta}(t,0)|\right]\cdot \left[\sup_{\tilde\theta\in\Theta,0\leq t\leq T}|\nabla_\theta\left((\nabla_x\bar\lambda^{\tilde\theta})(\bar X^{\tilde\theta}_t)\right)|\right]\cdot\left[\sup_{\tilde\theta\in\Theta,0\leq t\leq T}||\bar q^{\tilde\theta}(\bar X^{\tilde\theta}_t)||\right]\cdot\Delta\cdot|\theta_1-\theta_2|,
\end{align*}
where the difference of exponentials has been handled exactly as in the bound above for $|III|$; hence indeed there is a (perhaps larger) constant $K$ such that
\begin{align}
\left|Q_k(\theta_2)-Q_k(\theta_1)\right|&\leq K\cdot\Delta\cdot|\theta_1-\theta_2|,\nonumber
\end{align}
completing the proof of the lemma.

\qed

\subsection*{A.4 Lemmata for Frequent Sampling}
The contrast functions upon which our estimators are founded are defined as sums over a number of indices equal to the number of sampled data $n$. When $n$ is fixed, termwise convergence as $\epsilon\to0$ is enough to deduce convergence of the sums. When, however, $n$ is permitted to increase simultaneously as $\epsilon\to0$, one must examine the convergence of sums over increasingly many indices. Convergence is easily deduced if statements of termwise convergence can be strengthened so that they uniformly outpace the increasing number of indices. The strengthened bounds that allow us to take this approach are established in this section.

\begin{lemma}\label{nufbound} Assume Conditions \ref{basicconditions} and \ref{recurrencecondition} and, in the $\infty$ regime, Condition \ref{centeringcondition}. Assume that the sampling interval $\Delta:=T/n$ does not decrease too quickly relative to $\epsilon$; that is, suppose that there is a finite constant $G$ such that always $\epsilon\cdot n\leq G$. There is a constant $\tilde K$ such that for $\epsilon$ (and hence also $\delta=\delta(\epsilon))$ sufficiently small, independently of $(n,k)$,
\begin{align}
E|F^\varepsilon_k(\theta;\{X^{\epsilon, \theta}_{t_k}\}^n_{k=1})|^2\leq\tilde K\cdot\epsilon\cdot\Delta,\label{firstnufbound}\\
E|\tilde F_k(\theta;\{X^{\epsilon, \theta}_{t_k}\}^n_{k=1})|^2\leq\tilde K\cdot\epsilon\cdot\Delta\label{secondnufbound}.
\end{align}
\end{lemma}

\noindent\textit{Proof.} Recall that
\begin{align}
F^\varepsilon_k(\theta; \{X^{\epsilon, \theta}_{t_k}\}_{k=1}^n)&=\left[[X^{\epsilon, \theta}_{t_k}-\bar X^\theta_{{t_k}}]-Z^{\theta}(t_k,t_{k-1})\cdot[X^{\epsilon, \theta}_{t_{k-1}}-\bar X^\theta_{{t_{k-1}}}]\right]\label{nufboundreference}\\
&\hspace{2pc}-\sqrt\epsilon\int^{t_k}_{t_{k-1}}Z^{\theta}(t_k,s)\cdot\bar  J^\theta(\bar X^\theta_{s})ds\nonumber\\
&\hspace{2pc}=I+II+III,\nonumber
\end{align}
where
\begin{align*}
I&:=\int^{t_k}_{t_{k-1}}\left[\frac\epsilon\delta b(X^{\varepsilon,\theta}_s,Y^{\epsilon,\theta}_s)+c(X^{\varepsilon,\theta}_s,Y^{\epsilon,\theta}_s)-\bar\lambda^{\theta}(\bar X^{\theta}_s)\right]ds,\\
II&:=-\sqrt\epsilon\int^{t_k}_{t_{k-1}}\sigma(X^{\varepsilon,\theta}_s,Y^{\epsilon,\theta}_s)dW_s,\\
III&:=\left(1-Z^{\theta}(t_k,t_{k-1})\right)\cdot[X^{\varepsilon,\theta}_{t_{k-1}}-\bar X^{\theta}_{t_{k-1}}];
\end{align*}
it suffices to consider each of these separately.

$I$ is handled by arguments similar to those presented in the proofs of Theorem \ref{philimit} and Lemma \ref{moments}. It follows by the calculations (\ref{Phireference}) and, in the $\infty$ regime, (\ref{chireference}), that there is a constant $K$, which may be chosen independently of $(n,k)$, such that
\begin{align*}
E|I|^2&\leq \delta^2\cdot K\leq\epsilon\cdot\Delta\cdot K,
\end{align*}
where we have used the fact that $\delta$ is $O(\epsilon)$ and the assumption that $\epsilon$ is $O(\Delta)$.

Meanwhile, it is clear that there is a (perhaps larger) constant $K$ such that
\begin{align*}
E|II|^2\leq\epsilon\cdot\Delta\cdot K.
\end{align*}
Finally, in light of Theorem \ref{xlimit} and the fact that $||1_{m}-Z^{\theta_{2}}(t_k,t_{k-1})||$ is bounded by a finite constant times $\Delta$ as $\Delta\to 0$, it is clear that for a (perhaps larger) constant $K$,
\begin{align*}
E|III|^2&\leq\epsilon\cdot\Delta^2\cdot K\leq\epsilon\cdot\Delta\cdot T\cdot K.
\end{align*}

The gives the first statement, (\ref{firstnufbound}), in the $\infty$ regime. The proof for the $\gamma$ regime is similar and easier.

(\ref{secondnufbound}) may be obtained in the same way by omitting the last term in (\ref{nufboundreference}).

\qed

\begin{lemma}\label{nuflemma} Assume Conditions \ref{basicconditions} and \ref{recurrencecondition} and, in the $\infty$ regime, Condition \ref{centeringcondition}. Assume that the sampling interval $\Delta:=T/n$ does not decrease too quickly relative to $\epsilon$; that is, suppose that there is a finite constant $G$ such that always $\epsilon\cdot n\leq G$. There is a constant $\tilde K$ such that for $\epsilon$ (and hence also $\delta=\delta(\epsilon))$ sufficiently small, independently of $(n,k)$,
\begin{align}
\sup_{(\theta_1, \theta_2)\in\bar\Theta^2}E|F^\varepsilon_k(\theta_2;\{X^{\epsilon, \theta_1}_{t_k}\}^n_{k=1})|\leq\Delta\cdot\tilde K,\label{nufirstflemma}\\
\sup_{(\theta_1, \theta_2)\in\bar\Theta^2}E|\tilde F_k(\theta_2;\{X^{\epsilon, \theta_1}_{t_k}\}^n_{k=1})|\leq\Delta\cdot\tilde K.\label{nusecondflemma}
\end{align}
\end{lemma}

\noindent\textit{Proof.} Let us examine the proof of Lemma \ref{flemma}, with particular attention to (\ref{supxincrement}). It is not difficult to see that a constant $K$ may be found such that for $\epsilon$ sufficiently small,
\begin{align*}
EI&\leq K(\sqrt\epsilon+\sqrt\Delta)\sqrt\Delta\\
&\leq K(\sqrt{\epsilon\cdot n}+1)\Delta.
\end{align*}
Thus one sees that with the new assumption that $\epsilon\cdot n\leq G$, $EI$ is bounded for $\epsilon$ sufficiently small, independently of $n$, by a finite constant times $\Delta$. The same having been noted for $II+EIII+IV+V$, the desired bounds follow.

\qed

\begin{lemma}(Limit of the Contrast Function as $\epsilon+\Delta\to0$)\label{nucontrastlimit}
Assume Conditions \ref{basicconditions} and \ref{recurrencecondition} and, in the $\infty$ regime, Condition \ref{centeringcondition}. Assume that the sampling interval $\Delta:=T/n$ does not decrease too quickly relative to $\epsilon$, so that $\epsilon$ is $o(\Delta)$ as $\Delta\to0$. For any $\eta>0$,
\begin{align*}
\lim_{(\epsilon+\Delta)\to0}P\left(\sup_{\theta_1,\theta_2\in\Theta}|U^\varepsilon(\theta_2; \{X^{\epsilon,\theta_1}_{t_k}\}_{k=1}^n)-\tilde U(\theta_2; \{\bar X^{\theta_1}_{t_k}\}_{k=1}^n)|>\eta\right)=0.
\end{align*}
\end{lemma}

\noindent\textit{Proof.} Let us examine the proof of Lemma \ref{contrastlimit}. It will suffice to show that the terms $A^\varepsilon_k$, $B^\varepsilon_k$ vanish \textit{faster than $\Delta$} in probability uniformly in $\theta_1, \theta_2$ as $\epsilon+\Delta\to0$. Recall that Lemmata \ref{flemmabar} and \ref{qlemma} apply uniformly in $n$. Thus, (\ref{brh}) is bounded uniformly in $n$ and, with the new assumption that $\epsilon\cdot n\leq G$, one may replace Lemma \ref{flemma} with Lemma \ref{nuflemma} to establish the same for (\ref{arh}). Meanwhile, it is clear that (\ref{blh}) vanishes faster than $\Delta$. Thus, if we establish that (\ref{alh}) also vanishes faster than $\Delta$, we will have proven the lemma.

Splitting the right-hand side of (\ref{alh}), we obtain
\begin{align*}
F^\varepsilon_k(\theta_2;\{X^{\epsilon,\theta_1}_{t_k}\}_{k=1}^n)-F^\varepsilon_k(\theta_2;\{\bar X^{\theta_1}_{t_k}\}_{k=1}^n)&=I+II+III,
\end{align*}
where
\begin{align*}
I&:=\int^{t_k}_{t_{k-1}}\left[\frac\epsilon\delta b(X^{\epsilon,\theta_1}_s,Y^{\epsilon,\theta_1}_s)+c(X^{\epsilon,\theta_1}_s,Y^{\epsilon,\theta_1}_s)-\bar\lambda^{\theta_1}(\bar X^{\theta_1}_s)\right]ds,\\
II&:=\sqrt\epsilon\int^{t_k}_{t_{k-1}}\sigma(X^{\epsilon,\theta_1}_s,Y^{\epsilon,\theta_1}_s)dW_s,\\
III&:=\left(1_{m}-Z^{\theta_{2}}(t_k,t_{k-1})\right)\cdot[X^{\epsilon,\theta_1}_{t_{k-1}}-\bar X^{\theta_1}_{t_{k-1}}].
\end{align*}
$I$ is handled by an argument similar to the proof of Lemma \ref{moments} - letting $\chi$ be as in (\ref{chifunction}), one applies the It\^o formula to $\chi(X^{\varepsilon}_t,Y^{\varepsilon}_t)$ and rearranges the terms to see that there is a constant $K$, which may be chosen independently of $(n,k)$, such that
\begin{align*}
E|I|&\leq \delta\cdot K\leq\epsilon\cdot K,
\end{align*}
where we have used the fact that $\delta$ is $O(\epsilon)$.
Meanwhile, it is clear that there is likewise a (perhaps larger) constant $K$ such that
\begin{align*}
E|II|&\leq \sqrt{\epsilon\Delta}\cdot K.
\end{align*}
Finally, in light of Theorem \ref{xlimit} and the fact that $\left(1_{m}-Z^{\theta_{2}}(t_k,t_{k-1})\right)$ is bounded by a finite constant times $\Delta$, it is clear that for a (perhaps larger) constant $K$,
\begin{align*}
E|III|&\leq\sqrt\epsilon\Delta\cdot K.
\end{align*}
Hence,
\begin{align*}
E|F^\varepsilon_k(\theta_2;\{X^{\epsilon,\theta_1}_{t_k}\}_{k=1}^n)-F^\varepsilon_k(\theta_2;\{\bar X^{\theta_1}_{t_k}\}_{k=1}^n)|&\leq(\epsilon+\sqrt{\epsilon\Delta})\cdot 2^{3}\cdot K,
\end{align*}
which is enough since we have assumed $\epsilon$ is $o(\Delta)$.

\qed

\end{document}